\documentclass[final,a4paper,12pt]{amsart}
\usepackage{version}
\usepackage[notcite,notref]{showkeys}
\usepackage{amssymb}
\usepackage{amsfonts}
\usepackage{amsmath}                      
\usepackage{mathrsfs}
\usepackage{latexsym}           
\usepackage{amsthm}   
\usepackage{mathtools}

\usepackage{graphicx}
\graphicspath{ {./images/} }

\usepackage{blindtext}
\usepackage{tikz}
\usetikzlibrary{arrows}
\usepackage{tikz,pgf}
\usepackage{tikz-cd}
\usetikzlibrary{patterns,spy}
\usetikzlibrary{positioning}

\usepackage{color}
\colorlet{RED}{red}
\usepackage{epstopdf}
\usepackage[english]{babel}
\usepackage{float}
\usepackage{cite}

\usepackage{geometry}
\usepackage[utf8]{inputenc}            
\usepackage[T1]{fontenc}         
\usepackage{url}
\usepackage{hyperref}

\newcommand{\field}[1]{\mathbb{#1}}    
\newcommand{\N}{\field{N}}

\newcommand{\R}{\field{R}}

\newcommand{\mc}[1]{\mathcal{#1}}
\newcommand{\Pb}{\field{P}}
\newcommand{\E}{\field{E}}

\DeclareMathOperator{\dist}{\mathrm{dist}}
\DeclareMathOperator{\spt}{\mathrm{spt}}

\DeclareMathOperator{\dimh}{\mathrm{dim}_{\mathrm{H}}}
\DeclareMathOperator{\dimp}{\mathrm{dim}_{\mathrm{P}}}
\DeclareMathOperator{\dimhu}{\overline{\mathrm{dim}}_{\mathrm{H}}}
\DeclareMathOperator{\dimhl}{\underline{\mathrm{dim}}_{\mathrm{H}}}

\DeclareMathOperator{\dimBu}{\overline{\mathrm{dim}}_{\mathrm{B}}}
\DeclareMathOperator{\dimBl}{\underline{\mathrm{dim}}_{\mathrm{B}}}

\DeclareMathOperator{\dimMBu}{\overline{\mathrm{dim}}_{\mathrm{MB}}}
\DeclareMathOperator{\dimMBl}{\underline{\mathrm{dim}}_{\mathrm{MB}}}

\DeclareMathOperator{\dimMu}{\overline{\mathrm{dim}}_{\mathrm{M}}}
\DeclareMathOperator{\dimMl}{\underline{\mathrm{dim}}_{\mathrm{M}}}
\DeclareMathOperator{\dimloc}{\mathrm{dim}_{\mathrm{loc}}}
\DeclareMathOperator{\dimlocu}{\overline{\mathrm{dim}}_{\mathrm{loc}}}
\DeclareMathOperator{\dimlocl}{\underline{\mathrm{dim}}_{\mathrm{loc}}}
\DeclareMathOperator*{\esssup}{ess\,sup}
\DeclareMathOperator*{\essinf}{ess\,inf}
\DeclareMathOperator{\Jmurt}{\mathcal{J}_{\underline{r},t}^{\mu}}

\DeclareMathOperator{\Jmurz}{\mathcal{J}_{\underline{r},0}^{\mu}}
\DeclareMathOperator{\JmuVt}{\mathcal{J}_{\underline{V},t}^{\mu}}
\DeclareMathOperator{\JmuVz}{\mathcal{J}_{\underline{V},0}^{\mu}}

\DeclareMathOperator{\Leb}{\mathcal{L}_d}
\DeclareMathOperator{\muOne}{{\zeta}}
\DeclareMathOperator{\muTwo}{{\widetilde{\zeta}}}
\DeclareMathOperator{\ep}{\varepsilon}

\theoremstyle{plain}
\newtheorem{theorem}{Theorem}[section]      
\newtheorem{lemma}[theorem]{Lemma}      
\newtheorem{corollary}[theorem]{Corollary}      
\newtheorem{proposition}[theorem]{Proposition}

\newtheorem{conjecture}[theorem]{Conjecture}

\theoremstyle{definition}
\newtheorem{definition}[theorem]{Definition}

\theoremstyle{remark}
\newtheorem{remark}[theorem]{Remark}    

\numberwithin{equation}{section}

\begin{document}

\title[Ekstr\"om--Persson conjecture]{Hitting probabilities and the Ekstr\"om--Persson conjecture}

\author[E. J\"arvenp\"a\"a]{Esa J\"arvenp\"a\"a}
\address{Department of Mathematical Sciences, P.O. Box 3000, 90014 
University of Oulu, Finland}
\email{esa.jarvenpaa@oulu.fi}

\author[M. Myllyoja]{Markus Myllyoja}
\address{Department of Mathematical Sciences, P.O. Box 3000, 90014 
University of Oulu, Finland}
\email{markus.myllyoja@oulu.fi}

\author[S. Seuret]{St\'ephane Seuret}
\address{Universit\'e Paris-Est Cr\'eteil, Universit\`e Gustave Eiffel, CNRS, LAMA (UMR 8050), 
F-94010, Cr\'eteil, France}
\email{seuret@u-pec.fr}

\thanks{This research is partly funded by the Bézout Labex, funded by ANR, reference 
ANR-10-LABX-58, which supported the visit of EJ to Université Paris-Est Créteil. MM 
acknowledges the support of Emil Aaltonen Foundation.}


\begin{abstract}
We consider the Ekst\"om--Persson conjecture concerning the value of the Hausdorff
dimension of random covering sets formed by balls with radii $(k^{-\alpha})_{k=1}^\infty$ and 
centres chosen independently at random according to an arbitrary Borel probability measure 
$\mu$ on $\R^d$. The conjecture has been solved positively in the case 
$\frac 1\alpha\le\dimhu\mu$, where $\dimhu\mu$ stands for the upper Hausdorff dimension of 
$\mu$. In this paper, we develop a new approach in order to answer the full conjecture,
proving in particular  that the conjectured value is only a lower bound for the
dimension. Our approach opens the way to study more general limsup sets, and has consequences 
on the so-called hitting probability questions. For instance, we are able to determine whether 
and what part of a deterministic analytic set can be hit by random covering sets 
formed by open sets. 
\end{abstract}

\maketitle

\section{Introduction}\label{section:intro}
 
In this article, we settle a conjecture proposed by Ekstr\"om
and Persson in \cite{EP2018} regarding the Hausdorff dimension of the set of points that are 
covered infinitely many times by randomly distributed balls, using a  new approach that has 
consequences beyond this covering problem.
 
We start by fixing some notation.  In $\R^d$ endowed with the Euclidean norm, let $B(x,r)$ 
stand for the open ball and $\overline{B}(x,r)$ for the closed ball with centre
$x\in \R^d$ and radius $r>0$, respectively. We will use the notation
$\underline{r}:=(r_k)_{k=1}^{\infty}$ for sequences of positive numbers.
For $A\subseteq \R^d$, we denote by $\mathcal{P}(A)$ the space of compactly
supported Borel probability measures whose supports are contained in $A$ and denote the
support of a measure $\mu \in \mc{P}(\R^d)$ by 
$\spt\mu:=\{x \in\R^d \mid \mu(B(x,r))>0 {\text{ for all }}r>0\}$. 

\begin{definition}\label{definition:randomcoveringset}
Let $\mu\in\mathcal{P}(\R^d)$ and consider the probability space
$(\Omega,\mathbb{P})$, where $\Omega:=(\spt\mu)^{\N}$ and $\mathbb{P}:=\mu^{\N}$.
Let $\underline{r}$ be {a} sequence of positive numbers. Given
$\omega \in \Omega$, the \emph{random covering set} generated by the sequence
$\underline{r}$ is the limsup set
\[
E_{\underline{r}}(\omega)\coloneqq\limsup_{k\to\infty}B(\omega_{k},r_k)
  =\bigcap_{n=1}^{\infty}\bigcup_{k=n}^{\infty}B(\omega_{ k},r_k).
\]
\end{definition}

The study of random covering sets, or random limsup sets, finds its origins in the late 1890s
with Borel's seminal work on Taylor series \cite{Borel}, in which he noticed that a given
point on the circle $S^1$ almost surely belongs to infinitely many randomly placed 
arcs provided the
sum of their lengths is infinite. One major improvement of this result was obtained by
Shepp in \cite{Shepp} who  found a necessary and sufficient condition on the sequence 
$\underline{r}$ to satisfy that  $E_{\underline{r}}(\omega)= S^1$ almost surely when $\mu$ is 
the one-dimensional Lebesgue measure, answering a question by Dvoretzky \cite{Dvoretzky} (see 
also works by Fan and Kahane, e.g. \cite{Fan-Kahane,Fan2}).  
These  results, related to the Borel--Cantelli lemma, have inspired extensive research, 
which we detail hereafter.  

In this paper, we focus on the Hausdorff dimension of random covering sets, building upon
several works leading to a better understanding of the dimensional properties of such sets. 
First, observe that, by the Kolmogorov's zero-one law, the Hausdorff 
dimension $\dimh E_{\underline{r}}(\omega)$ is constant $\mathbb{P}$-almost surely
since $\{\omega\in\Omega\mid\dimh E_{\underline{r}}(\omega)\le\beta\}$ is
a tail event for all $\beta\ge 0$ (see e.g. \cite[Lemma~3.1]{JJKLSX2017} for the measurability 
arguments). We denote by $f_{\mu}(\underline{r})$ this almost sure value of
$\dimh E_{\underline{r}}(\omega)$.

\smallskip

The Hausdorff dimension of the limsup sets  $E_{\underline{r}}(\omega)$ relies on the 
multifractal properties of the measure $\mu$ used to draw the
balls $B(\omega_k,r_k)$. Let us recall some of the multifractal concepts for measures (similar 
quantities can be defined for describing the pointwise regularity of functions). The lower 
and upper local dimensions of $\mu$ at a point $x\in \R^d$ are defined by
\[
\dimlocl \mu(x):=\liminf_{r\to 0}\frac{\log \mu(B(x,r))}{\log r}
\,\text{ and }\,
\dimlocu \mu(x):=\limsup_{r\to 0}\frac{\log \mu(B(x,r))}{\log r}.
\]
Then the lower and upper Hausdorff dimensions of $\mu$ are
\[
\dimhl \mu:=\essinf_{x\sim \mu}\dimlocl \mu(x)
\,\text{ and }\,
\dimhu \mu:=\esssup_{x\sim \mu}\dimlocl \mu(x).
\]

\begin{definition}
Let $\mu \in \mathcal{P}(\R^d)$. For every $t\geq 0$, let 
\begin{equation}
\label{def-Dmut}
D_\mu(t):= \{ x \in \spt \mu\mid \dimlocl \mu (x)\leq t \}.
\end{equation}
The lower Hausdorff and packing multifractal spectra of the lower local dimension of $\mu$  
are defined by 
\begin{align*}
F_{\mu}(t)& := \dimh D_\mu(t),\\
H_{\mu}(t)& :=\dimp D_\mu(t),
\end{align*}
where $\dimh$ and $\dimp$ stand for the Hausdorff and packing dimensions, respectively.
\end{definition}

The connection between multifractals and random limsup sets
has been a key issue for a long time. For instance, in order to study the  regularity 
of lacunary wavelet series in \cite{Jaffard-lac} or Lévy processes in \cite{Jaffard-levy},  
Jaffard needed to compute the (almost sure) Hausdorff dimension of sets
$E_{\underline{r}}(\omega)$ for uniformly distributed balls with random radii $\underline{r}$. 
To achieve his goal, he was led to prove  in \cite{Jaffard-levy} a first version of what is 
nowadays called, after the fundamental article by Beresnevich and Velani \cite{BV}, a 
{\em mass transference principle}, which allows to compute the dimension of limsup 
coverings by balls in $\R^d$ (see also \cite{BS-Ubiq,Daviaud,E-B2024,KR} for other versions of 
the mass transference principle). Applying his deterministic theorem to random balls, Jaffard
found the almost sure Hausdorff dimension of the set of those points covered infinitely often
by uniformly distributed balls of radii $\underline{r}$, this dimension depending on the decay 
rate $s_2(\underline{r})$ of $\underline{r}$ to 0 defined by
\begin{equation}\label{def-s2}
s_2(\underline{r}):=\inf\Bigl\{ t>0\mid \sum_{k=1}^\infty r_k^{t}<\infty\Bigr\}.
\end{equation}
It is easily seen that $s_2(\underline{r})$ is an upper bound for 
$\dimh E_{\underline{r}}(\omega)$ for any realisation $\omega$ by observing
that, for every $N\in \N$,
\[
E_{\underline{r}}(\omega)\subseteq \bigcup_{k=N}^{\infty}B(\omega_k,r_k).
\]
Jaffard's result was later refined  for balls of the form $B(\omega_k,k^{-\alpha})$  by Fan
and Wu \cite{Fan-Wu} and Durand \cite{Durand} using slightly different approaches. This is one 
illustration of the increasing role of  multifractal analysis in various mathematical fields, 
see \cite{ViklundLawler,BS-Levy,Shmerkin,JaffardMartin,BarralFeng,BanicaVega} for some other 
recent examples.

There are several directions in which these results can be generalised. 
A first natural extension consists in replacing balls by more general sets in $\R^d$, still  
distributed according to the $d$-dimensional Lebesgue measure. Let us mention \cite{JJKLS}
that deals with uniformly distributed affine rectangles and \cite{Persson} 
where Persson generalises the previously obtained bound (which relies on the 
singular value function associated with the rectangles) for open sets. Finally, the Hausdorff 
dimension of random covering obtained as a limsup of randomly distributed general Lebesgue 
measurable sets 
was completely determined in \cite{FJJV} when $\mu$ has a component  absolutely continuous 
with respect to the $d$-dimensional Lebesgue measure. There are also connections with hitting 
probability questions, which we explore hereafter.

The second direction consists in drawing balls, rectangles or open sets in spaces different 
than $\R^d$: this has been studied in \cite{JJKLSX2017} where balls in metric 
spaces were considered, in \cite{EJJ2020,Myllyoja1} for rectangles in Heisenberg 
groups, in \cite{EJJV2018} for rectangles in products of Ahlfors regular spaces and in 
\cite{E-B2024} for open sets in Ahlfors regular unimodular groups.

A third possibility consists in replacing the random sequence 
$(\omega_n)_{n=1}^\infty$ by orbits $(T^n(x))_{n=1}^\infty$ of 
well-chosen $\mu$-invariant dynamical systems $T\colon X\to X$, following the idea that 
sufficiently mixing dynamical systems give rise to (typically with respect to $\mu$) well 
distributed orbits. This is closely related to the famous shrinking targets problem, see 
\cite{Hill-Velani, FST, Liao-Seuret,Liao-Velani-Zorin} among many references in this active    
research field.

The last extension, which we mainly consider in this article, focuses on $\R^d$ but in
the situation where the balls are distributed according to a general probability 
measure $\mu$ that is purely singular with respect to the Lebesgue
measure and, for instance, may have Cantor-like support. We will consider the special case 
\begin{equation}\label{defralpha}
\underline r=\underline r(\alpha):=(k^{-\alpha})_{ k=1}^\infty
\end{equation}
for $\alpha>0$, and 
we write $E_{\alpha}(\omega)$ for the corresponding limsup set, that is,
\[
E_{\alpha}(\omega):=\limsup_{k\to\infty}B(\omega_{k},k^{-\alpha}).
\]
We also use the short hand notation
\begin{equation}\label{def-fmualpha}
f_{\mu}(\alpha):= \mbox{almost sure value of } \dimh E_{\alpha}(\omega).
\end{equation}

In \cite{Seuret}, Seuret considered the situation where $\mu$ is a Gibbs measure on the 
symbolic space, demonstrating that, in this case, the dimension $f_\mu(\alpha)$ indeed depends 
on the "standard" multifractal spectrum of $\mu$. Then Ekstr\"om and 
Persson in \cite{EP2018} considered balls drawn randomly according to a general measure on 
$\R^d$ and conjectured that $f_\mu(\alpha)$ is equal to the value 
$\overline{F}_\mu(\frac{1}{\alpha})$ of the increasing 1-Lipschitz hull of the lower 
multifractal spectrum $F_\mu(\frac{1}{\alpha})$ of $\mu$. Recall that 
if $A\subseteq\mathopen[0,\infty\mathclose[$ and $g\colon A\to\R$ is bounded from above, the
increasing $1$-Lipschitz hull $\overline g\colon\mathopen[0,\infty\mathclose[\to\R$ of the
function $g$ is defined by setting
\begin{equation}
\label{defgbar}
\overline{g}(t):= \inf \{ h(t)\mid h\geq g \text{ is increasing and } 1\text{-Lipschitz 
  continuous} \}.
\end{equation}

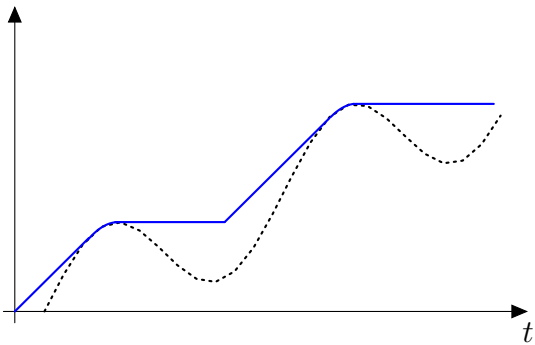
\begin{figure}[ht!]
\begin{center}
        \begin{tikzpicture}[line cap=round, line join=round, >=triangle 45,
                            x=1.0cm, y=1.0cm, scale=1.5]
                           
      \draw [->,color=black] (-0.1,0) -- (4.5,0);
      \draw [->,color=black] (0,-0.1) -- (0,2.7);
      \draw [thick,dotted, domain=0:4]  plot ({\x+0.26},{\x/2+sin(3*deg(\x))/2});
      \draw [thick, domain=0:0.65, color=blue]  plot ({\x},{\x});
      \draw [thick, domain=0.37:0.65, color=blue]  plot ({\x+0.27},{\x/2+sin(3*deg(\x))/2}); 
      \draw [thick, domain=0.88:1.84, color=blue]  plot ({\x},{0.79});          
      \draw [thick, domain=0.79:1.71, color=blue]  plot ({\x+1.05},{\x});
      \draw [thick, domain=2.5:2.7, color=blue]  plot ({\x+0.27},{\x/2+sin(3*deg(\x))/2});
      \draw [thick, domain=2.97:4.2, color=blue]  plot ({\x},{1.835});          
      \draw (4.5,0) node [below] {$t$};

      \end{tikzpicture}
\end{center}
\caption{Illustration of a function $g$ and its increasing 1-Lipschitz hull $\overline{g}$
   plotted in dotted black and blue, respectively.}
\label{figure:lipshitzhull}
\end{figure}

The following conjecture regarding $f_\mu(\alpha)$ was stated in \cite{EP2018}:

\begin{conjecture}[Ekström,Persson]\label{conjecture:EP}
    For every $\alpha>0$, 
    \[
    f_{\mu}(\alpha)=\overline{F}_{\mu}\Bigl(\frac 1\alpha\Bigr){.}
    \]
\end{conjecture}

In \cite[Theorem~2.3]{JJMS2025}, it was shown that Conjecture~\ref{conjecture:EP} is at least 
partly true:

\begin{theorem}[Järvenpää, Järvenpää, Myllyoja, Stenflo]\label{theorem:nminusalpha}
Let $\mu\in \mc{P}(\R^d)$ and $\alpha>0$. If $\frac{1}{\alpha}\le\dimhu\mu$,
then 
\[
f_{\mu}(\alpha)=\overline{F}_{\mu}\Bigl(\frac 1\alpha\Bigr){=}\frac{1}{\alpha} .
\]
\end{theorem}

Theorem~\ref{theorem:nminusalpha} proves Conjecture~\ref{conjecture:EP} in the case 
$\frac 1\alpha \leq \dimhu \mu$, i.e.{,} when the random balls are quite small. Note that 
the upper bound $f_\mu(\alpha)\leq \frac{1}{\alpha}$ is immediate by using a trivial covering 
 $\{B(\omega_k,k^{-\alpha})\}_{k=N}^\infty$ for $E_{\alpha}(\omega)$, and observe 
that the equality $\overline{F}_{\mu}(\frac 1\alpha)=\frac 1\alpha$ above
for $\frac 1\alpha \leq \dimhu \mu$ 
follows by combining the well-known fact that $F_{\mu}(t)\leq t$ for any $t$ with the fact 
that $F_{\mu}(\dimhu \mu)\geq \dimhu \mu$. 

Let us also mention that, in the region $s_2(\underline r)\le\dimhu\mu$, precise 
estimates regarding $f_\mu(\underline{r})$ were obtained in 
\cite[Theorem 2.5]{JJMS2025} for general sequences of radii $\underline{r}$, in terms of
$s_2(\underline{r})$ and the $\mu$-essential values of the local dimensions of $\mu$.

The remaining open question is thus to examine the validity of Conjecture \ref{conjecture:EP} 
in the case where $\frac 1\alpha > \dimhu \mu$, i.e.{,} when the random balls have large
diameters. In this situation, it is expected that some overlaps between these larger balls  
will occur, yielding a dimension $f_\mu(\alpha)$ smaller than $\frac1\alpha$. 

The following bounds for $f_{\mu}(\alpha)$ are proved in \cite{EP2018}:

\begin{theorem}[Ekström, Persson]\label{theorem:EPThm2}
For every $\alpha>0$, it is true that
\[
\lim_{s\to\frac 1\alpha -}F_{\mu}(s)\leq f_{\mu}(\alpha)\leq \min \Bigl\{\frac{1}{\alpha},
 \overline{H}_{\mu}\Bigl(\frac 1\alpha\Bigr) \Bigr\}.
\]
\end{theorem}

In \cite{EP2018}, Ekström and Persson also obtained an upper bound for $f_{\mu}(\alpha)$ involving the upper coarse spectrum of $\mu$, but we choose not to define it here.

In this paper, we will show that Conjecture \ref{conjecture:EP} is not true for all 
$\alpha>0$, but the conjectured value nevertheless always gives a lower bound for the
dimension. Our main theorem is the following:

\begin{theorem}\label{theorem:dimhatleastconjecture}
For every $\mu \in \mc{P}(\R^d)$ {and} for every $\alpha>0$, one has 
\begin{equation}\label{bounds-fmu}
\overline{F}_{\mu}\Bigl(\frac 1\alpha\Bigr) \leq f_{\mu}(\alpha)\leq \min\Bigl\{
  \frac{1}{\alpha}, \overline{H}_{\mu}\Bigl(\frac 1\alpha\Bigr)\Bigr\},
\end{equation}
and the above bounds are optimal.
\end{theorem}

In particular, as soon as 
$\overline{F}_{\mu}(\frac 1\alpha)=\overline{H}_{\mu}(\frac 1\alpha)$ (which is a condition
different from
$F_{\mu}(\frac 1\alpha)=H_{\mu}(\frac 1\alpha)$), the three quantities agree and 
$f_{\mu}(\alpha)=\overline{F}_{\mu}{(\frac 1\alpha)}$.
In particular, we recover 
Theorem~\ref{theorem:nminusalpha} with another approach.
The optimality of \eqref{bounds-fmu} is proved by constructing
examples of measures saturating the inequalities (see 
 Sections~\ref{example:dimensionishausdorffspectrum} and 
\ref{example:dimensionispackingspectrum}). 

\medskip

The main issues in Theorem~\ref{theorem:dimhatleastconjecture} are the lower bound and the
optimality in \eqref{bounds-fmu}.
Let us now explain our approach, which differs from the previous ones on such problems and can 
certainly be used for dealing with more general random limsup sets and hitting probability 
questions. Our method consists in explicitly characterising when a random limsup set will 
touch (almost surely) a given compact set, and is based on the following operator defined on 
compact sets of $\mathbb{R}^d$. For a set $A\subseteq \R^d$ and $r>0$, we write $A(r)$ for the 
$r$-neighbourhood of $A$, that is,
\[
A(r)\coloneqq \{x\in \R^d\mid \dist (x,A)<r\},
\]
 where $\dist$ refers to the Euclidean distance.

\begin{definition}
For a positive sequence $\underline{r}=(r_k)_{k=1}^\infty$, define the hitting operator 
$\Jmurz$ on compact sets $K\subset\R^d$ by setting
\begin{equation}\label{defJt}
\Jmurz(K):=\Bigl\{x\in \R^d\mid\sum_{k=1}^\infty \mu\bigl(B(x,r)\cap K(r_k)\bigr) 
  =\infty\text{ for every }r>0\Bigr\}.
\end{equation}
\end{definition}
We will prove (see Lemma~\ref{lem-Jmurt}) that for every compact set $K$, $\Jmurz(K)$ is a
subset of $K$ that  remains compact (which may be empty). Starting with any compact $K$, 
iterating $\Jmurz$ will ultimately yield a fixed point: 

\begin{proposition}\label{intro-ordinal}
 \label{lemma:Itildeinvariantsetfoundaftercountablymanysteps-intro}
Let $\mu \in \mc{P}(\R^d)$, let $\underline{r}$ be a positive nonincreasing sequence, and  
let $C\subseteq \spt \mu$ be compact. Define compact sets $C_0^{\gamma}$ for
each ordinal number $\gamma$ as follows: 
\begin{enumerate}
  \item[i)] Set $C_0^0:=C$.
  \item[ii)] If $C_0^{\gamma}$ has been defined, set 
             $C_0^{\gamma+1}:=\Jmurz(C_0^{\gamma})$.
  \item[iii)] If $\lambda$ is a nonzero limit ordinal and $C_0^{\gamma}$ has been defined for 
              every $\gamma<\lambda$, set 
        \[
        C_0^{\lambda}:=\bigcap_{\gamma<\lambda}C_0^{\gamma}.
        \]
\end{enumerate}        
Then there exists a countable ordinal $\beta_0$ such that 
$C_0^{\beta_0+1}=C_0^{\beta_0}$.
\end{proposition}

A remarkable property is that this set $C_0^{\beta_0}$ actually governs the possibility of 
$E_{\underline{r}}(\omega)$ to intersect the initial deterministic compact set $C$. Indeed, 
the following theorem, which should be understood as a hitting probabilities result, 
characterises completely the event that $E_{\underline{r}}(\omega)$ touches $C$, solely in 
terms of whether $C_0^{\beta_0}$ is empty or not. This justifies the name {\em hitting} operator that we use for $\Jmurz$.

\begin{theorem}\label{lemma:ifinvariantsetemptythennointersection-intro}
Let $\mu \in \mc{P}(\R^d)$, let $\underline{r}$ be a positive nonincreasing sequence and let 
$C\subseteq \spt \mu$ be compact. Let $\beta_0$ be the ordinal number given by
Proposition~\ref{intro-ordinal}. Almost surely, one has 
$C\cap E_{\underline{r}}(\omega)\subset C_0^{\beta_0}$ and 
\[
C\cap E_{\underline{r}}(\omega)
           \begin{cases} 
                  =\emptyset&\text{ if } \, C_0^{\beta_0}=\emptyset,\\
                  \ne\emptyset&\text{ if }\,  C_0^{\beta_0}\ne\emptyset.
           \end{cases}       
\]
\end{theorem}

This actually follows from the fact that any nonempty compact set which is a fixed point for 
$\Jmurz$ will be hit by $E_{\underline{r}}(\omega)$ almost surely (see 
Propositions~\ref{lemma:conditionforpositivehittingprobability}  and 
\ref{lemma:necessaryoconditionfordimensionofintersection} below). Finally, for every $\alpha$, 
the lower bound in \eqref{bounds-fmu}
will be a consequence of a fine study of the dimensional properties of well-chosen compact 
sets $K$ included in $D_\mu(\frac{1}{\alpha})$ that are fixed points for some generalised 
versions of the operator $\Jmurz$ studied in Section \ref{section:hittingprob}.
 
Building on the study of the operator $\Jmurz$, we are able to improve some dimensional 
results concerning the hitting probabilities problem, which has already been studied in many
papers and culminated in \cite{JJKLSX2017} (see \cite{HLX} for corresponding results
concerning limsup sets determined by balls distributed along orbits of exponentially mixing
dynamical systems). There, the  authors studied the probability that limsup 
sets of random balls hit analytic sets $F$ in Ahlfors regular metric spaces, in terms of the 
packing and Hausdorff dimensions of $F$ (see Theorem \ref{theorem:JJKLSX_hittingprob}  below 
for their explicit result). Our approach allows to recover this result with a different method
and to replace the Hausdorff dimension by the modified lower box counting
dimension $\dimMBl$ for compact sets.

\begin{theorem}\label{theorem:hittingprobforregularmeasure}
Let $X:=[0,1]^d$ and denote by $\Leb$ the restriction of the $d$-di\-men\-sion\-al Lebesgue 
measure to $X$. Let $\underline{r}$ be a sequence of radii such that 
$s_2(\underline{r})\leq d$. Then, for  analytic sets $C\subseteq X$, 
\begin{enumerate}
\item[i)] if 
          $\dimp C<d- s_2(\underline{r})$, then $C\cap E_{\underline r}(\omega)=\emptyset$
          almost surely,
 \item[ii)] if 
           $\dimh C>d- s_2(\underline{r})$, then 
           $C\cap E_{\underline{r}}(\omega)\neq\emptyset$ almost surely,        
\item[iii)] if  $C$ is compact and
           $\dimMBl C>d- s_2(\underline{r})$, then 
           $C\cap E_{\underline{r}}(\omega)\neq\emptyset$ almost surely,
\item[iv)]  if 
             $\underline r=\underline r{(\alpha)}$  for some  
            $\alpha>0$ and $\dimp C>d-\frac 1\alpha$, then 
            $C\cap E_{ \alpha}(\omega)\neq \emptyset$ almost surely.
\end{enumerate}
\end{theorem}

 \smallskip

The paper is organised as follows.

In Section~\ref{section:hittingprob}, we study an operator $\JmuVt$ on compact sets which is 
adapted not only to limsup sets generated by balls but more generally to limsup sets generated 
by randomly distributed open sets $\underline{V}:=(V_k)_{k=1}^\infty$. We prove  versions of 
Proposition~\ref{lemma:Itildeinvariantsetfoundaftercountablymanysteps-intro} and 
Theorem~\ref{lemma:ifinvariantsetemptythennointersection-intro} for this larger class of   
limsup sets and operators (see 
Proposition~\ref{lemma:Itildeinvariantsetfoundaftercountablymanysteps} and 
Corollary~\ref{lemma:ifinvariantsetemptythennointersection}).

The inequality \eqref{bounds-fmu} of Theorem~\ref{theorem:dimhatleastconjecture} is proved in 
Section~\ref{section:lowerbound}. Actually, we prove a slightly stronger version of it in 
Theorem~\ref{thm:dimhofintersectionwithsublevelset}.

The optimality of the bounds in 
Theorem~\ref{theorem:dimhatleastconjecture} is obtained in Section~\ref{section:examples} 
by constructing two examples of measures $\mu$ such that the Hausdorff and 
packing spectra of $\mu$ do not agree  and 
$\overline{F}_{\mu} (\frac 1\alpha )<\overline{H}_{\mu} (\frac 1\alpha )$ at some exponents 
$\alpha$, one satisfying  
$f_\mu(\alpha)=\overline{F}_{\mu} (\frac 1\alpha )$ and the other one   
$f_\mu(\alpha)=\overline{H}_{\mu} (\frac 1\alpha )$.

 Theorem~\ref{theorem:hittingprobforregularmeasure} is proved in 
Section~\ref{sec-hitting}.

Finally, Section~\ref{sec-conclusion} concludes with some remarks on the necessity of 
considering ordinals in 
Proposition~\ref{lemma:Itildeinvariantsetfoundaftercountablymanysteps-intro}, and with some 
extensions and perspectives.

\section{Hitting operators $\JmuVt$}
\label{section:hittingprob}
  
In this section, we study the following question: Given a measure $\mu \in \mc{P}(\R^d)$, a 
compact set $C\subseteq \spt \mu$ and a sequence of open sets 
$\underline{V}:=(V_k)_{k=1}^{\infty}$, and defining
\[
E_{\underline{V}}(\omega)\coloneqq\limsup_{k\to\infty} (\omega_k+V_k)  
=\bigcap_{n=1}^{\infty}\bigcup_{k=n}^{\infty}(\omega_k +V_k) ,
\]
under what conditions does the event 
$\left\{ C\cap E_{\underline{V}}(\omega)\neq\emptyset\right\}$ have probability 
$1$ and when is this probability equal to zero? This question includes the setting presented 
in the introduction, where $V_k=B(0,r_k)$.

We will study this question essentially in full generality, assuming only that the diameters 
of the open sets tend to $0$ (which is the interesting case), and that the origin is contained 
in each of the open sets (to ensure that the translates $\omega_k+V_k$ intersect the support 
of the measure).

\subsection{The hitting operators $\JmuVt$ on compact sets and their fixed points}

We write $|A|$ for the diameter of a set $A\subseteq \R^d$, and for an open set 
$V\subseteq \R^d$, let
\[
A-V\coloneqq \{a-v\mid a\in A \text{ and } v\in V \}
\]
stand  for the difference set of $A$ and $V$.
Observe that the difference set $A-B(0,r)$ coincides with the $r$-neighbourhood $A(r)$ of $A$. 

\begin{definition}
A sequence $\underline{V}=(V_k)_{k=1}^{\infty}$ of open subsets of $\R^d$ is called admissible 
if $0\in V_k$ for all $k$ and $|V_k|\to 0$ as $k\to \infty$. 
\end{definition}
 
Given $\mu\in\mc P(\R^d)$, let $X$ be a compact cube containing $\spt\mu$. Let $\mc{K}$ denote 
the space of compact subsets of $X$. The following operators, key in our proof, generalise 
the one introduced in the introduction only for balls  and for the parameter value $t=0$. 

\begin{definition}
Given an admissible sequence of open sets $\underline V$, define, for all $t\ge 0$, the hitting 
operator $ \JmuVt\colon \mc{K}\to \mc{K}$ by setting
\begin{equation}\label{defJVt}
\JmuVt(K):=\Bigl\{x\in X\mid\sum_{k=1}^\infty \mu\bigl(B(x,r)\cap (K-V_k)\bigr) |V_k|^t
  =\infty\text{ for every }r>0\Bigr\}.
\end{equation}
\end{definition}

In the case when $t=0$, $\underline{r}:=(r_k)_{k=1}^{\infty}$ and $V_k=B(0,r_k)$ for each $k$, 
the above definition reduces to \eqref{defJt}. In particular, results below will imply 
Proposition~\ref{lemma:Itildeinvariantsetfoundaftercountablymanysteps-intro} and 
Theorem~\ref{lemma:ifinvariantsetemptythennointersection-intro} as special cases.

Observe that it may happen that $\JmuVt(K)=\emptyset$. Let us justify that 
$\JmuVt\colon \mc{K}\to \mc{K}$.

\begin{lemma}
\label{lem-Jmurt}
For every $K\in \mc{K}$, $\JmuVt(K)\in\mc{K}$ and 
\begin{equation}\label{Jtinclusion}
\JmuVt(K)\subseteq K.
\end{equation}
\end{lemma}

\begin{proof}
First, $\JmuVt(\emptyset)=\emptyset$. Next, for a nonempty $K\in \mc {K}$, we argue as
follows: If $x\not\in \JmuVt(K)$, there exists 
$r>0$ such that 
\[
\sum_{k=1}^\infty \mu\bigl(B(x,r)\cap (K-V_k)\bigr){|V_k|}^t<\infty.
\]
Then, for every $y\in B(x,\frac r2)$, we have that $B(y,\frac r2)\subseteq B(x,r)$, hence 
\[
\sum_{k=1}^\infty \mu\bigl(B(y,\tfrac r2)\cap (K-V_k)\bigr){|V_k|}^t<\infty.
\]
This implies that $B(x,\frac r2)\subseteq \R^d\setminus \JmuVt(K)$, hence the
complement of $\JmuVt(K)$ is open. As $\JmuVt(K)\subseteq X$, it is compact.

Regarding \eqref{Jtinclusion}, the condition 
\begin{equation}
\label{cond-mu}
\sum_{k=1}^\infty\mu\bigl(B(x,r)\cap (K-V_k)\bigr){|V_k|}^t=\infty
\end{equation}
implies that $B(x,r)\cap (K-V_k)\neq\emptyset$ for infinitely many $k$ (and, thus, for all 
$k$ since the sets $V_k\ni 0$ are open  and ${|V_k|}\to 0$), hence $x\in K((1+\varepsilon)r)$ for all $\varepsilon>0$ (since ${|V_k|}\to 0$). If this holds for all $r>0$, 
then $x\in K$ since $K$ is compact. 
\end{proof}

It is also true that if 
\[
\sum_{k=1}^\infty\mu(K-V_k){|V_k|}^t=\infty
\]
then $\JmuVt(K)\neq \emptyset$, but we do not need to explicitly use this fact.

As explained in the introduction, we iterate the operator $\JmuVt$ to get a fixed point.

\begin{proposition}\label{lemma:Itildeinvariantsetfoundaftercountablymanysteps}
Let $\mu \in \mc{P}(\R^d)$, let $\underline{V}$ be an admissible sequence of open sets and let 
$C\subseteq \spt \mu$ be compact. For all $t\ge 0$, define compact sets $C_t^{\gamma}$ for
each ordinal number $\gamma$ as follows: 
\begin{enumerate}
  \item[i)] Set $C_t^0:=C$.
  \item[ii)] If $C_t^{\gamma}$ has been defined, set 
             $C_t^{\gamma+1}:=\JmuVt(C_t^{\gamma})$.
  \item[iii)] If $\lambda$ is a nonzero limit ordinal and $C_t^{\gamma}$ has been defined for 
              every $\gamma<\lambda$, set 
        \[
        C_t^{\lambda}:=\bigcap_{\gamma<\lambda}C^{\gamma}_t.
        \]
\end{enumerate}        
Then there exists a countable ordinal $\beta_t$ for which $C_t^{\beta_t}$ is a fixed point of 
$\JmuVt$, i.e.,  $C_t^{\beta_t+1}=C_t^{\beta_t}$.
\end{proposition}

\begin{remark}
We will prove in Section~\ref{sec-ordinals} that it is necessary to consider all 
 countable ordinals to reach a compact set that is a fixed point.
\end{remark}

\begin{proof}
The proof is similar to that of the fact that every closed subset of $\R^d$ is a union 
of a perfect set and a countable set using ordinal numbers.    
Let $\{U_n\}_{n=1}^{\infty}$ be a countable basis for the topology of $\R^d$ and, for compact 
sets $K\subseteq \R^d$, define 
\[
N(K)\coloneqq \{ n\in \N\mid U_n\cap K\neq \emptyset \}.
\]
Then the complement of a compact set $K$ can be written as 
$\bigcup_{n\in\N\setminus N(K)}U_n$. This implies that if $K$ and $K'$ are two compact sets
such that $K\subsetneq K'$, then $N(K)\subsetneq N(K')$. Fix $t\ge 0$. Since the sets
$C_t^{\gamma}$ are nested by \eqref{Jtinclusion} and compact, we have that 
$N(C_t^{\gamma+1})\subseteq N(C_t^{\gamma})$ for each ordinal number $\gamma$. Since 
$N(C_t^{0})\subseteq \N$ and the set of all countable ordinals is uncountable, the strict
inclusion $N(C_t^{\gamma+1})\subsetneq N(C_t^{\gamma})$ cannot hold for every countable
ordinal $\gamma$. Hence, there is a countable $\beta_t$ such that 
$N(C_t^{\beta_t+1})=N(C_t^{\beta_t})$, which yields the claim.
\end{proof}

\subsection{Fixed points of $\JmuVz$ intersect $E_{\underline{V}}(\omega)$}

The following proposition shows that those compact sets that are fixed points of 
the operator $\JmuVz$ intersect the random covering set $E_{\underline{V}}(\omega)$ almost
surely. In the next subsection, we will see that this condition
characterises which kind of compact sets are hit by the random covering sets almost surely.

\begin{proposition}\label{lemma:conditionforpositivehittingprobability}
Let $\mu \in \mc{P}(\R^d)$ and let $\underline{V}$ be an admissible 
sequence of open sets. Suppose that $C\subseteq \spt \mu$ is a nonempty compact set such that  
$C=\JmuVz(C)$. Then $C\cap E_{\underline{V}}(\omega)\neq \emptyset$ almost surely.
\end{proposition}

\begin{proof}
Let us rewrite the fact that $C=\JmuVz(C)$ as 
\begin{equation}\label{equation:sufficientcondiitonfornonemptyintersection}
C=\Bigl\{x\in C\mid\sum_{k=1}^\infty\mu(B(x,r)\cap (C-V_k))=\infty \text{ for all }r>0\Bigr\}.
\end{equation}
By scaling, we may assume that $\spt\mu\subset X:=[1/4,3/4]^d$. Since $C$ is compact, there 
exists $\nu\in \mc{P}(X)$ such that $\spt \nu=C \subset \spt \mu$ (for example,
one can take $\nu:=\sum_{n=1}^\infty 2^{-n}\delta_{c_n}$, where $\{c_n\mid n\in \N\}$ is a
countable dense subset of $C$). We denote the boundary of a set $A$
by $\partial A$. Let $\mc{D}:=\bigcup_{n=0}^{\infty}\mc{D}_n$ be
a translated collection of half-open dyadic cubes where the translation is
chosen such that $\nu(\partial D)=0$ for every $D\in \mc{D}$. Note that there are at most
countably many hyperplanes parallel to 
the axes with positive $\nu$-measure, hence such a translation exists and, further, we may 
assume that $X\subset D\in\mc{D}_0$.
Thus, for each $D\in \mc{D}_n$ with $\nu(D)>0$, $D$ contains some ball with centre in $C$. By 
\eqref{equation:sufficientcondiitonfornonemptyintersection}, 
\[
\sum_{k=1}^\infty \mu(D\cap (C-V_k))=\infty.
\]
Since, for each $n\in\N$, there are only finitely many $D\in \mc{D}_n$, we can choose
sequences $(M_n)_{n=1}^\infty$ and $(N_n)_{n=1}^\infty$ of natural numbers in such a way that
$M_n<N_n<M_{n+1}$ for every $n\in\N$ and
\[
\sum_{k=M_n}^{N_n}\mu(D\cap(C-V_k))\geq 2^n
\]
for every $D\in \mc{D}_n$ with $\nu(D)>0$.

Define, for each $n\in \N$, a measure
\[
\nu_n\coloneqq \sum_{D\in \mc{D}_n}\nu(D)a_{n,D}\sum_{k=M_n}^{N_n}\mu_{|_{D\cap (C-V_k)}},
\]
where 
\[
0<a_{n,D}\coloneqq \left(\sum_{k=M_n}^{N_n}\mu(D\cap (C-V_k)) \right)^{-1}\leq 2^{-n}
\]
if $\nu(D)>0$ and $a_{n,D}:=0$ otherwise. Notice that $\nu_n(D)=\nu(D)$ for every 
$D\in \mc{D}_n$, hence $\nu_n\to \nu$ as $n\to \infty$ in the weak-$\ast$ topology.

We write $\int h\, d\nu$ or $\nu(h)$ for the integral of a measurable function $h$ with 
respect to a measure $\nu$.

\begin{lemma}\label{sublemma}
For $\omega \in \Omega$ and $n\in \N$, define a measure $\mu_n^{\omega}$ by
\[
\mu_n^{\omega}\coloneqq \sum_{D\in \mc{D}_n}\nu(D)a_{n,D}\sum_{k=M_n}^{N_n} 
  \chi_{D\cap (C-V_k)}(\omega_k )\delta_{\omega_k}.
\]
Then $\mu_n^{\omega}\to \nu$ almost surely in the weak-$\ast$ topology.
\end{lemma}

\begin{proof}[Proof of Lemma \ref{sublemma}] Let $h\in C(X)$.
Then 
\begin{align*}
\mathbb{E}(\mu_n^{\omega}(h))&=\sum_{D\in \mc{D}_n}\nu(D)a_{n,D}\sum_{k=M_n}^{N_n} 
  \int_{D\cap (C-V_k)}h(\omega_k)\, d\mu(\omega_k) = \int h\, d\nu_n\xrightarrow[n\to\infty]{}
  \int h\,d\nu.
\end{align*}
We also estimate the variance of $\mu_n^{\omega}(h)$:
\begin{align*}
\mathrm{Var} (\mu_n^{\omega}(h))&
   =\mathrm{Var} \left( \sum_{D\in \mc{D}_n}\nu(D)a_{n,D}\sum_{k=M_n}^{N_n} 
    \chi_{D\cap (C-V_k)}(\omega_k ) h(\omega_k) \right)\\
 &=\sum_{k=M_n}^{N_n}\mathrm{Var} \left(  \sum_{D\in \mc{D}_n}\nu(D)a_{n,D}\,
   \chi_{D\cap (C-V_k)}(\omega_k ) h(\omega_k) \right)\\
 &\leq \Vert h \Vert_{\infty}^2\sum_{k=M_n}^{N_n}\E \left(\left(\sum_{D\in \mc{D}_n}
   \nu(D)a_{n,D}\,\chi_{D\cap(C-V_k)}(\omega_k )  \right)^2  \right)\\
 &=\Vert h \Vert_{\infty}^2\sum_{k=M_n}^{N_n} \sum_{D'\in \mc{D}_n}\int_{D'}\left(\sum_{D\in 
   \mc{D}_n}\nu(D)a_{n,D}\,\chi_{D\cap (C-V_k)}(\omega_k )  \right)^2\, d\mu(\omega_k)\\
 &=\Vert h \Vert_{\infty}^2\sum_{k=M_n}^{N_n} \sum_{D'\in \mc{D}_n}\int_{D'}
   \bigl(\nu(D')a_{n,D'}\,\chi_{D'\cap (C-V_k)}(\omega_k )\bigr)^2\, d\mu(\omega_k)\\
 &=\Vert h \Vert_{\infty}^2\sum_{k=M_n}^{N_n} \sum_{D'\in \mc{D}_n}\nu(D')^2a_{n,D'}^2
   \mu(D'\cap (C-V_k))\\
 &=\Vert h \Vert_{\infty}^2\sum_{D'\in \mc{D}_n}\nu(D')^2a_{n,D'}
   \leq \Vert h \Vert_{\infty}^2 2^{-n}.
\end{align*}
For every $\varepsilon>0$, recalling that $\mathbb{E}(\mu_n^{\omega}(h)) =\nu_n(h)$, we obtain 
by the Chebyshev's inequality that
\[
\sum_{n=1}^\infty\Pb (|\mu_n^{\omega}(h)-\nu_n(h) |\geq \varepsilon)\leq \sum_{n=1}^\infty 
\varepsilon^{-2}\mathrm{Var}(\mu_n^{\omega}(h))<\infty.
\]
Thus, the Borel--Cantelli lemma implies that 
\[
\Pb (\limsup_{n\to \infty}|\mu_n^{\omega}(h)-\nu_n(h)|\geq \varepsilon)=0.
\]
Since $\lim_{n\to\infty}\nu_n(h)=\nu(h)$, we obtain that 
$\lim_{n\to\infty}\mu_n^{\omega}(h)=\nu(h)$ almost surely and, by the separability of $C(X)$,
$\mu_n^{\omega}\xrightarrow[n\to\infty]{\text{weak-}\ast} \nu$ almost surely.
\end{proof}

With Lemma~\ref{sublemma} established, we can now complete the proof of 
Proposition~\ref{lemma:conditionforpositivehittingprobability}. To show that the intersection 
$C\cap E_{\underline{V}}(\omega)$ is nonempty, we are going to build a probability measure 
$\eta$ whose support is included in this intersection. This will be achieved thanks to the     
following lemma, which is a simplified version of  \cite[Lemma~1.4]{EJJ2020} 
and whose demonstration is included in its proof (recall that the support of a continuous
function $\varphi$ is $\spt\varphi:=\overline{\{x{\mid}\varphi(x)\neq 0\}}$).

\begin{lemma}\label{EJJlemma}   
Let $\nu$ be a finite Borel measure on a compact metric space $X$ and let
$(\varphi_n)_{n=1}^\infty$ be a sequence of nonnegative continuous functions on
$X$ with the property that $\lim_{n\to \infty}\varphi_n\,d\nu=\nu$ in the
weak-$\ast$ topology. Then there exists $\eta\in\mc{P}(X)$ such that  
\[
\spt\eta\subset \spt\nu\cap(\limsup_{n\to\infty}\spt\varphi_n).
\]
\end{lemma}

We are going to build positive continuous functions supported inside $\omega_k+V_k$  when 
$\omega_k\in C-V_k$. To this end, for each $k\in\N$ and  
$\omega=(\omega_n)_{n=1}^\infty\in \Omega$, set 
$\tilde\psi_k^{\omega} = \psi_k^{\omega}\equiv 0$ (the zero function) when 
$\omega_k \not\in C-V_k$.
Otherwise, when $\omega_k\in C-V_k$, the open set $\omega_k+V_k$ contains some ball with
centre in $C$, hence $\nu(\omega_k+V_k)>0$. For each $k\in\N$, for every
$\omega\in \Omega$ such that $\omega_k \in C-V_k$  and for all $r>0$, define
$(\omega_k+V_k)[r]:=\{x\in\omega_k+V_k\mid\dist\bigl(x,(\omega_k+V_k)^c\bigr)\ge r\}$,
$r(\omega_k):=\max\{r>0\mid\nu\bigl((\omega_k+V_k)[r]\bigr)\ge\frac 12\nu(\omega_k+V_k)\}$ and
$U_k^\omega:=(\omega_k+V_k)[r(\omega_k)]$. Here the superscript $c$ refers to the complement.
Then, let $\Tilde{\psi}_k^{\omega }$ be a continuous function such that
\[
\chi_{U_k^{\omega}}\leq \Tilde{\psi}_k^{\omega}\leq \chi_{\omega_k+V_k}
\]
and $(\omega,x) \in \Omega\times X \mapsto \Tilde{\psi}_k^{\omega}(x)$ is a Borel map. 
 The existence of such a map can be checked along the same lines as \cite[(4.12)]
{JJMS2025}, where balls instead of open sets $V_k$ are considered. 
Then, for every $\omega_k\in C-V_k$, the probability
measure $\psi_k^{\omega}\, d\nu$ is supported on $C\cap (\omega_k+V_k)$, where 
\[
\psi_k^{\omega}(x)\coloneqq \frac{\Tilde{\psi}_k^{\omega}(x)}
{\int_{\omega_k+V_k}\Tilde{\psi}_k^{\omega}\, d\nu}.
\]
Finally, let 
\[
\varphi_n^{\omega}\coloneqq \sum_{D\in \mc{D}_n}\nu(D)a_{n,D}\sum_{k=M_n}^{N_n} 
\chi_{D\cap (C-V_k)}(\omega_k ) \psi_k^{\omega}.
\] 
Since $\mu_n^{\omega}\xrightarrow[n\to\infty]{\text{weak-}\ast} \nu$ almost surely and 
$\lim_{k\to\infty}{|V_k|}=0$, we obtain that 
$\varphi_n^{\omega}\,d\nu\xrightarrow[n\to\infty]{\text{weak-}\ast} \nu$ almost surely. 
Observe also that $\spt \varphi_n^\omega \subset  \bigcup_{k=M_n}^{N_n} (\omega_k+V_k)$, so  
$\limsup_{n\to \infty} \spt \varphi_n^\omega \subset E_{\underline{V}}(\omega)$.

By Lemma~\ref{EJJlemma}, there exists almost surely a probability measure supported on 
$C\cap (\limsup_{n\to \infty}\spt \varphi_n^{\omega})$ and, in particular, 
$C\cap E_{\underline{V}}(\omega)\neq \emptyset$ almost surely. This completes 
the proof of Proposition~\ref{lemma:conditionforpositivehittingprobability}.
\end{proof}

\subsection{Dimension estimates for intersection with compact sets}
 
Let us recall that, for $t\geq0$, the $t$-dimensional Hausdorff content of a set $E$ is
\[
\mc{H}_{\infty}^t(E)= \inf\Big\{ \sum_{i\in I} | U_i|^t \mid E \subset 
\bigcup_{i\in I} U_i \Big\},
\]
where the infimum is taken over all possible coverings $\{U_i\}_{i\in I}$ of $E$. 
The following proposition is our main tool in
the proof of {Theorem~\ref{theorem:dimhatleastconjecture} and its refined version  
Theorem~\ref{thm:dimhofintersectionwithsublevelset}}. It states in 
particular that, given a compact set $C$, the covering set 
$E_{\underline{V}}(\omega)$ may intersect $C$ only inside $C_0^{\beta_0}$.

\begin{proposition}\label{lemma:necessaryoconditionfordimensionofintersection}
Let $\mu \in \mc{P}(\R^d)$, let $\underline{V}$ be an admissible sequence of open sets and let
$C\subseteq \spt \mu$ be compact. For every $t\geq 0$, let $\beta_t$ be the ordinal number given 
by Proposition~\ref{lemma:Itildeinvariantsetfoundaftercountablymanysteps}. Then, for every
$t\ge 0$, 
\[
\dimh (E_{\underline{V}}(\omega)\cap (C\setminus C_t^{\beta_t})) \leq t
\]
almost surely. Furthermore, 
\[
E_{\underline{V}}(\omega)\cap (C\setminus C_0^{\beta_0})=\emptyset
\]
almost surely.
\end{proposition}

\begin{proof}
Fix $t\ge 0$. We will first show that, for any compact set $K$, almost surely
\begin{align}
\dimh\bigl(E_{\underline{V}}(\omega)\cap (K\setminus\JmuVt (K))\bigl)
   &\leq t \ \ \text{ for }t>0\label{equation:dimofcomplementofJ1lessthant}\\
\text{ and }  \ \ E_{\underline{V}}(\omega)\cap (K\setminus\JmuVz(K))&=\emptyset.
    \label{empty} 
\end{align}
Observe first that 
\begin{align*}
K\setminus\JmuVt (K)&=\Bigl\{ x\in K\mid \sum_{k=1}^\infty
    \mu\bigl(B(x,r)\cap (K-V_k)\bigr){|V_k|}^t<\infty \text{ for some } r>0\Bigr\}\\
  &=\bigcup_{n=1}^\infty\Bigl\{x\in K\mid \sum_{k=1}^\infty\mu\bigl(B(x,\tfrac 1n)\cap 
    (K-V_k)\bigr){|V_k|}^t<\infty\Bigr\}\eqqcolon \bigcup_{n=1}^{\infty} W_n.
\end{align*}
By the countable stability of Hausdorff dimension, to get
\eqref{equation:dimofcomplementofJ1lessthant} it suffices to show that 
$\dimh(W_n\cap E_{\underline{V}}(\omega))\leq t$ almost surely for every $n\in\N$.
To this end, suppose that $x\in W_n$. Then, since $\lim_{k\to\infty}{|V_k|}=0$, we have that 
\begin{equation}\label{equation:conditionforWn(t)}
\sum_{k=1}^\infty \mu\bigl((\overline{B}(x,\tfrac 1{2n})-V_k)\cap 
(K-V_k)\bigr){|V_k|}^t<\infty.
\end{equation}
Notice now that if 
$(\omega_k+V_k)\cap \overline{B}(x,\frac 1{2n})\cap K\neq \emptyset$, then 
\[
\omega_k\in (\overline{B}(x,\tfrac 1{2n})\cap K)-V_k\subseteq 
(\overline{B}(x,\tfrac 1{2n})-V_k)\cap (K-V_k).
\]
Since $\mc{H}_{\infty}^t(\omega_k+V_k)\leq {|V_k|}^t$ always holds, we obtain
\[
\mc{H}_{\infty}^t\bigl((\omega_k+V_k)\cap \overline{B}(x,\tfrac 1{2n})\cap K\bigr)
\leq \chi_{(\overline{B}(x,\frac 1{2n})-V_k)\cap (K-V_k)}(\omega_k){|V_k|}^t. 
\]
Thus, for any $N$,
\begin{align*}
\E\Bigl( \mc{H}_{\infty}^t\bigl(E_{\underline{V}}(\omega)\cap \overline{B}(x,\tfrac 1{2n})
    \cap K\bigr) \Bigr)&
    \leq \sum_{k=N}^{\infty}\E\Bigl( \mc{H}_{\infty}^t\bigl((\omega_k+V_k)\cap
    \overline{B}(x,\tfrac 1{2n})\cap K\bigr)\Bigr)\\
  &\leq \sum_{k=N}^{\infty}\E\Bigl( \chi_{(\overline{B}(x,\frac 1{2n})-V_k)\cap (K-V_k)}
    (\omega_k){|V_k|}^t\Bigr)\\
  &= \sum_{k=N}^{\infty}\mu\bigl((\overline{B}(x,\frac 1{2n})-V_k)\cap (K-V_k)\bigr){|V_k|}^t.
\end{align*}
By \eqref{equation:conditionforWn(t)}, the last sum tends to zero as $N\to \infty$. Hence, 
\[
\E \left( \mc{H}_{\infty}^t(E_{\underline{V}}(\omega)\cap \overline{B}(x,\tfrac 1{2n})\cap K) 
\right)=0.
\]
This implies that 
$\mc{H}_{\infty}^t\bigl(E_{\underline{V}}(\omega)\cap \overline{B}(x,\frac 1{2n})\cap 
K\bigr)=0$ 
almost surely, giving 
\begin{align*}
\dimh\bigl(E_{\underline{V}}(\omega)\cap \overline{B}(x,\tfrac 1{2n})\cap K)
&\leq t \ \ \text{ for }t>0\\
\text{ and }\ \ E_{\underline{V}}(\omega)\cap \overline{B}(x,\tfrac 1{2n})\cap K 
&=\emptyset\ \ \text{ for }t=0
\end{align*}
almost surely. By covering the set $W_n\subseteq K$ by finitely many closed balls of radius
$\frac 1{2n}$ and centres in $W_n$, we obtain the claims 
\eqref{equation:dimofcomplementofJ1lessthant} and \eqref{empty}. 

For the rest of the proof, we assume that $t>0$. The case $t=0$ can be done in a similar 
manner. We will show by transfinite induction that, for every countable ordinal $\lambda$, 
\begin{equation}\label{equation:inductionclaimforJt}
  \dimh \bigl(E_{\underline{V}}(\omega)\cap (C\setminus C_t^{\lambda})\bigr)\leq t  
\end{equation}
almost surely. This implies the desired claim, since $\beta_t$ is countable (recall 
Proposition~\ref{lemma:Itildeinvariantsetfoundaftercountablymanysteps}). The case 
$\lambda=0$ is trivial and the case $\lambda=1$ is already established by
\eqref{equation:dimofcomplementofJ1lessthant}. Suppose that, for some countable ordinal 
$\lambda$, 
\[
\dimh \bigl(E_{\underline{V}}(\omega)\cap (C\setminus C_t^{\gamma})\bigr)\leq t
\]
almost surely for every $\gamma<\lambda$. Then, if $\lambda$ is a successor ordinal, there is 
an ordinal number $\gamma$ such that $\lambda=\gamma+1$ and
$C_t^{\lambda}=\JmuVt (C_t^{\gamma})\subseteq C_t^{\gamma}$. We have that 
\begin{align*}
E_{\underline{V}}(\omega)\cap(C\setminus C_t^{\lambda})&=(E_{\underline{V}}(\omega)
  \cap(C\setminus C_t^{\gamma}))\cup (E_{\underline{V}}(\omega)\cap(C_t^{\gamma}\setminus 
  C_t^{\lambda}))\\
&=(E_{\underline{V}}(\omega)\cap(C\setminus C_t^{\gamma}))\cup (E_{\underline{V}}(\omega)  
  \cap(C_t^{\gamma}\setminus \Jmurt(C_t^{\gamma})).  
\end{align*}
The first set in the union has Hausdorff dimension at most $t$ almost surely by the induction
hypothesis \eqref{equation:inductionclaimforJt}  and since $C_t^{\gamma}$ is compact, the 
latter set in the union has Hausdorff
dimension at most $t$ almost surely by \eqref{equation:dimofcomplementofJ1lessthant}. Thus,
\eqref{equation:inductionclaimforJt} holds for successor ordinals.

If $\lambda$ is a limit ordinal, then $C_t^{\lambda}=\bigcap_{\gamma<\lambda}C_t^{\gamma}$, 
whence
\begin{align*}
\dimh \bigl(E_{\underline{V}}(\omega)\cap (C\setminus C_t^{\lambda})\bigr)
  &=\dimh\biggl(\bigcup_{\gamma<\lambda}\bigl(E_{\underline{V}}(\omega)\cap (C\setminus 
    C_t^{\gamma})\bigr)\biggr)\\
  &=\sup_{\gamma<\lambda}\dimh\bigl(E_{\underline{V}}(\omega)\cap 
    (C\setminus C_t^{\gamma})\bigr)\leq t
\end{align*}
almost surely since $\lambda$ is countable. By transfinite induction,
\eqref{equation:inductionclaimforJt} holds for every countable ordinal $\lambda$ and the proof 
of the lemma is complete.
\end{proof}

The following corollary gives a necessary and a sufficient condition depending only on
$C_0^{\beta_0}$ for the event 
$C\cap E_{\underline{V}}\neq \emptyset $ to happen almost surely and, together with 
\eqref{empty}, immediately implies 
Theorem~\ref{lemma:ifinvariantsetemptythennointersection-intro}:

\begin{corollary}\label{lemma:ifinvariantsetemptythennointersection}
Let $\mu \in \mc{P}(\R^d)$, let $\underline{V}$ be an admissible sequence of open
sets and let $C\subseteq \spt \mu$ be compact. Let $\beta_0$ be the ordinal number given by
Proposition~\ref{lemma:Itildeinvariantsetfoundaftercountablymanysteps}. Then, almost surely,
\[
C\cap E_{\underline{V}}(\omega)
           \begin{cases} 
                  =\emptyset&\text{ if } \, C_0^{\beta_0}=\emptyset,\\
                  \ne\emptyset&\text{ if }\,  C_0^{\beta_0}\ne\emptyset.
           \end{cases}       
\]
\end{corollary}

\begin{proof}
The first claim follows from 
Proposition~\ref{lemma:necessaryoconditionfordimensionofintersection} and the second one
from Proposition~\ref{lemma:conditionforpositivehittingprobability} applied to 
$C_0^{\beta_0}$.
\end{proof}

One may ask whether a counterpart of
Corollary~\ref{lemma:ifinvariantsetemptythennointersection} is valid also  when 
$C\subseteq\spt \mu$ is a Borel set or an analytic set. Since the neighbourhood of a set and 
that of its closure coincide, the definition of the hitting operator 
$\Jmurt$ (or $\JmuVt$) will not distinguish between a set and its closure, whilst the 
questions whether  
the random covering set intersects a set or its closure are drastically different. However,
the following proposition guarantees that  in the context of the hitting probability problem, 
analytic sets can be approximated from the inside by compact sets, implying that 
Corollary~\ref{lemma:ifinvariantsetemptythennointersection} can also be applied to study 
analytic sets.

\begin{proposition}\label{compacthitting}
Let $\mu \in \mc{P}(\R^d)$ and let $\underline{V}$ be an admissible sequence of open sets.
Assume that $A\subset\R^d$ is an analytic set such that 
$A\cap E_{\underline{V}}(\omega)\ne\emptyset$ almost surely. Then there exists a compact 
set $C\subset A$ such that $C\cap E_{\underline{V}}(\omega)\ne\emptyset$ almost surely.
\end{proposition}

\begin{proof}
Consider the set 
$G:=\{(\omega,x)\in\Omega\times A\mid x\in E_{\underline{V}}(\omega)\}
 \subset\Omega\times\R^d$. 
 Since
it is analytic, there is a uniformising function $f\colon\pi_1(G) \to A$ that is
measurable with respect to the $\sigma$-algebra generated by analytic sets by 
\cite[Theorem 18.1]{K1995}. Here $\pi_1\colon\Omega\times A\to\Omega$ is the
projection onto the first coordinate.
Define $\nu:=f_\ast\Pb$. Then $\nu$ is a Borel probability measure with $\nu(A)=1$  since
$\Pb(\pi_1(G))=1$. As all Borel probability measures on $\R^d$ are inner regular, there
exists a compact set $C\subset A$ with $\nu(C)>0$. Now
\[
0<\nu(C)=\Pb(\{\omega\in\Omega\mid f(\omega)\in C\})\le\Pb(\{\omega\in\Omega\mid
   C\cap E_{\underline{V}}(\omega)\ne\emptyset\}).
\] 
Since $C\cap E_{\underline{V}}(\omega)\ne\emptyset$ is a tail event, the claim follows.
\end{proof}

\section{Proof of inequality \eqref{bounds-fmu} in 
         Theorem~\ref{theorem:dimhatleastconjecture} } \label{section:lowerbound}
    
In this section, we prove Theorem~\ref{thm:dimhofintersectionwithsublevelset}, which concerns 
a general measure $\mu\in \mathcal{P}(\R^d)$ and a fixed sequence 
of radii $\underline r(\alpha) =(k^{-\alpha})_{k=1}^\infty$ and is slightly stronger than
Theorem~\ref{theorem:dimhatleastconjecture}.

\begin{theorem}\label{thm:dimhofintersectionwithsublevelset}
Let $\mu\in \mathcal{P}(\R^d)$ and $t>0$ with $D_\mu(t) \neq \emptyset$. Then, almost surely, for every $\alpha>0$ with $t-F_\mu(t)<\frac 1\alpha \leq t$, we have that
\begin{equation}
\label{eq-min-dim-Ealpha}
\dimh (E_{\alpha}(\omega)\cap {D_\mu(t)})\geq \frac{1}{\alpha}+F_\mu(t)-t.
\end{equation}
\end{theorem} 

Note that Theorem~\ref{thm:dimhofintersectionwithsublevelset} holds trivially if
$D_\mu(t)=\emptyset$ or $F_\mu(t)=0$.

\begin{proof}[Proof that Theorem~\ref{thm:dimhofintersectionwithsublevelset} implies 
              Theorem~\ref{theorem:dimhatleastconjecture}]
For a mapping $g\colon[0,\infty[\to \R$, let us introduce a function 
$\Tilde{g}\colon\mathopen[0,\infty\mathclose[\to\R$ by setting 
$\Tilde{g}(t):=t+\sup_{y\geq t}(g(y)-y)$. Then one easily checks that $\Tilde{g}\leq \overline{g}$ (recall \eqref{defgbar}), and equality holds if $g$ is non-decreasing.

Since $F_{\mu}$ is non-decreasing, we have that 
\[
\overline{F}_{\mu}\Bigl(\frac 1\alpha\Bigr)=\widetilde{F}_{\mu}\Bigl(\frac 1\alpha\Bigr)
=\frac 1\alpha+\sup_{y\geq\frac 1\alpha}(F_{\mu}(y)-y).
\]
Suppose that $\overline{F}_{\mu}(\frac 1\alpha)>0$. Let $\varepsilon>0$ be small and
choose $y\geq \frac 1\alpha$ such that 
\[
\frac 1\alpha +F_{\mu}(y)-y\geq \overline{F}_{\mu}\Bigl(\frac 1\alpha\Bigr)-\varepsilon>0.
\]
Note that this implies that $F_{\mu}(y)>0$. In particular, 
$D_\mu(y)\ne\emptyset$. Since $y-F_{\mu}(y)<\frac 1\alpha \leq y$, 
we have, by Theorem~\ref{thm:dimhofintersectionwithsublevelset}, that almost surely 
\[
f_{\mu}(\alpha)\geq \dimh (E_{\alpha}(\omega)\cap D_\mu({y}))\geq \frac 1\alpha+F_\mu(y)-y\geq 
\overline{F}_{\mu}\Bigl(\frac 1\alpha\Bigr)-\varepsilon.
\]
Letting $\varepsilon\to 0$ through a countable sequence yields the claim.
\end{proof}

In order to prove Theorem~\ref{thm:dimhofintersectionwithsublevelset}, we need a few lemmas. 
Let us recall that a measure $\mu\in \mathcal{P}(\R^d)$ is called an $s$-Frostman
measure, if there is a constant $C>0$ such that $\mu(B(x,r))\le Cr^s$ for all $r>0$ and 
$x\in\R^d$.

We will make use of the following lemma, which is proved in \cite{EP2018} (there the authors
needed this fact only for $(F_\mu(t)-\varepsilon)$-Frostman measures supported on 
$D_\mu(t)$, but the same proof works for any Borel probability measure).

\begin{lemma}\label{lemma:EP7.1}
If $\alpha,t>0$ are such that $t<\frac 1\alpha$, then, for any 
$\nu \in \mc{P}(D_\mu(t))$, 
\[
\nu(E_{\alpha}(\omega))=1
\]
almost surely.
\end{lemma}

\begin{proof}
See the proofs of \cite[Lemmas 7.1 and 7.2]{EP2018}.
\end{proof}

\begin{lemma}\label{lemma:transferinformationfromalpha0toalpha}
Fix $t>0$. Suppose that $\nu \in \mc{P}(D_\mu(t))$ is $t_0$-Frostman for 
some $0<t_0\leq F_\mu(t)$. Write $K_\nu :=\spt \nu$ and recall that 
$\underline r(\alpha)=(k^{-\alpha})_{k=1}^\infty$. The following properties hold:
\begin{enumerate}
  \item[1)] for every $\varepsilon >0$ and every $\alpha>0$ with   
            $\frac 1{\alpha}>t$,  $K_\nu= \mathcal{J}^{\mu}_{\underline r(\alpha),t_0-
            \varepsilon}(K_\nu)$.
  \item[2)] for  every $\alpha>0$ such that $\frac 1{\alpha}>t-t_0$,
       $K_\nu=\mathcal{J}^{\mu}_{\underline r(\alpha),0}(K_\nu)$.
\end{enumerate}
\end{lemma}

 \begin{proof}
We first prove 1). To this end, fix $\varepsilon>0$ and $\frac 1{\alpha}>t$. Recall 
that $K_\nu\subseteq D_\mu(t)$ is compact and that $t_0\leq F_\mu(t) \leq t$. By 
Lemma~\ref{lemma:EP7.1}, $\nu(E_{\alpha}(\omega))=1$
almost surely. Since $\nu$ is $t_0$-Frostman, $\dimhl \nu\geq t_0$ and, in particular, 
$\nu(A)=0$ for any Borel set $A$ with $\dimh A<t_0$. Consider the set (recall \eqref{defJVt})
\begin{align*}
\widetilde K_\nu &\coloneqq\mathcal{J}^{\mu}_{\underline r(\alpha),t_0-\varepsilon}(K_\nu)\\
&=\Bigl\{ x\in K_\nu\mid \sum_{k=1}^\infty
   \mu\bigl(B(x,r)\cap K_\nu(k^{-\alpha})\bigr)k^{-\alpha(t_0-\varepsilon)}=\infty 
   \text{ for all } r>0\Bigr\}.
\end{align*}
Then $\widetilde K_\nu\subseteq K_\nu$ is compact by Lemma~\ref{lem-Jmurt}.
Furthermore, by \eqref{equation:dimofcomplementofJ1lessthant},
\[
\dimh\bigl((K_\nu\setminus{\widetilde K_\nu}\bigr)\cap E_{\alpha}(\omega))\leq t_0-\varepsilon
\]
almost surely. This implies that 
$\nu((K_\nu\setminus\widetilde K_\nu)\cap E_{\alpha}(\omega))=0$ 
almost surely. Since $\nu(E_{\alpha}(\omega))=1$ almost surely, we must then have that 
$\nu(\widetilde K_\nu\cap E_{\alpha}(\omega))=1$ with probability $1$. This can 
only happen if $\nu(\widetilde K_\nu)=1$. Thus, 
$\widetilde K_\nu\subseteq K_\nu=\spt \nu$ is a compact set with full 
$\nu$-measure, hence $\widetilde K_\nu=\spt \nu=K_\nu$, i.e., $K_\nu$ is a 
fixed point for $\mathcal{J}^{\mu}_{\underline r(\alpha),t_0-\varepsilon}$. 
This  proves 1).

To prove 2), suppose that $\frac 1\alpha>t-t_0$. Fix $\varepsilon>0$ small enough so that 
$t-t_0+3\varepsilon<\frac 1\alpha$, and let $\alpha_0>0$ be such that 
$t<\frac 1{\alpha_0}<t+\varepsilon$. Let $x\in K_\nu$ and $r>0$. By part 1), 
$K_\nu = \mathcal{J}^{\mu}_{\underline r(\alpha_0),t_0-\varepsilon}(K_\nu)$ and we have that
\[
\sum_{k=1}^\infty\mu\bigl(B(x,r)\cap K_\nu(k^{-\alpha_0})\bigr)k^{-\alpha_0(t_0-\varepsilon)}
=\infty.
\]
This implies that 
\begin{equation}\label{equation:lowerboundforalpha0}
\mu\bigl(B(x,r)\cap K_\nu(k^{-\alpha_0})\bigr)k^{-\alpha_0(t_0-\varepsilon)}
  \geq  k^{-(1+\varepsilon\alpha_0)}=k^{-\alpha_0(\frac 1{\alpha_0}+\varepsilon)}
\end{equation}
for infinitely many $k\in \N$, since otherwise the sum above would be finite.
For each $k$ such that \eqref{equation:lowerboundforalpha0} holds, define 
$g(k)\coloneqq \lfloor k^{\frac{\alpha_0}{\alpha}}\rfloor$, where $\lfloor x\rfloor$ is 
the integer part of $x\in\R$. Then, for every large enough 
$k\in\N$ (depending on $\alpha$ and $\alpha_0$), we have that 
\[
2^{\alpha_0}k^{-\alpha_0}\geq (k^{\frac{\alpha_0}{\alpha}}-1)^{-\alpha}
  \geq g(k)^{-\alpha}\geq k^{-\alpha_0}.
\]
Thus, for each large $k\in\N$ satisfying \eqref{equation:lowerboundforalpha0} (note that
$K_\nu(r)\subseteq K_\nu(r')$ if $r\leq r'$), we have that 
\begin{align*}
\mu\bigl(B(x,r)\cap K_\nu(g(k)^{-\alpha})\bigr)&
     \geq \mu\bigl(B(x,r)\cap K_\nu(k^{-\alpha_0})\bigr)
     \geq k^{-\alpha_0(\frac 1{\alpha_0}+\varepsilon)}k^{\alpha_0(t_0-\varepsilon)}\\
   &=k^{-\alpha_0(\frac 1{\alpha_0}-t_0+2\varepsilon)}
     \geq \delta g(k)^{-\alpha(\frac 1{\alpha_0}-t_0+2\varepsilon)}, 
\end{align*}
where $\delta\coloneqq 2^{-\alpha_0(\frac 1{\alpha_0}-t_0+2\varepsilon)}>0$. Since 
$g(k)\to \infty$ as $k\to \infty$, we obtain that 
\begin{equation}\label{equation:lowerbounforalphaobtainedfromalpha0}
  \mu\bigl(B(x,r)\cap K_\nu(j^{-\alpha})\bigr)\geq \delta 
  j^{-\alpha(\frac 1{\alpha_0}-t_0+2\varepsilon)}
\end{equation}
for infinitely many $j\in \N$. Notice now that if $j\in \N$ satisfies
\eqref{equation:lowerbounforalphaobtainedfromalpha0} and $\ell \in \N$ is such that 
$\frac j2\leq \ell \leq j$, then
\[
\mu\bigl(B(x,r)\cap K_\nu(\ell ^{-\alpha})\bigr)\geq\mu\bigl(B(x,r)\cap K_\nu(j^{-\alpha})
\bigr) \geq \delta j^{-\alpha(\frac 1{\alpha_0}-t_0+2\varepsilon)}\geq \delta' 
\ell ^{-\alpha(\frac 1{\alpha_0}-t_0+2\varepsilon)},
\]
where $\delta'=\delta'(\delta, \alpha,\alpha_0,t_0,\varepsilon)>0$. Thus, for any $j\in\N$ 
satisfying \eqref{equation:lowerbounforalphaobtainedfromalpha0}, we have that
\[
\sum_{\ell = \lceil\frac j2\rceil}^j\mu\bigl(B(x,r)\cap K_\nu(\ell ^{-\alpha})\bigr)
\geq \delta' \sum_{\ell =\lceil\frac j2\rceil}^j \ell ^{-\alpha(\frac 1{\alpha_0}-t_0+2\varepsilon)}
\geq \frac{1}{2}\delta' j^{1-\alpha(\frac 1{\alpha_0}-t_0+2\varepsilon)}.
\]
Recalling now that $\frac{1}{\alpha_0}<t+\varepsilon$, we conclude that 
\[
\frac 1{\alpha_0}-t_0+2\varepsilon<t-t_0+3\varepsilon<\frac 1\alpha.
\]
Thus, $j^{1-\alpha(\frac 1{\alpha_0}-t_0+2\varepsilon)}\to \infty$ as $j\to \infty$, and 
since \eqref{equation:lowerbounforalphaobtainedfromalpha0} is satisfied for infinitely many
$j\in \N$, we must have that 
\begin{equation}
\label{sum-infinite}
\sum_{j=1}^{\infty}\mu\bigl(B(x,r)\cap K_\nu(j^{-\alpha})\bigr)=\infty.
\end{equation}
Since $x\in K_\nu$ and $r>0$ were arbitrary, this shows that 
$K_\nu =\mathcal{J}^{\mu}_{\underline r(\alpha) ,0}(K_\nu)$
which completes the proof of part $2)$.
\end{proof}

\begin{lemma}\label{lemma:codimensionargumentunlocked}
Let $\alpha,t>0$ be such that $t-F_\mu(t) <\frac 1\alpha<t$. 
For any compact $K\subseteq D_\mu(t)$, if $\dimh K>t-\frac{1}{\alpha}$ then 
$K\cap E_{\alpha}(\omega)\neq \emptyset$ almost surely.
\end{lemma}

\begin{proof}
Set $\gamma:= t-\frac{1}{\alpha}$. Let $\varepsilon>0$ and let $K\subseteq D_\mu(t)$ be a 
compact set with $\dimh K>\gamma + \varepsilon\eqqcolon t_0$. Then there exists a 
$t_0$-Frostman measure $\nu \in \mc{P}(K)$. Let $K_\nu:=\spt \nu\subseteq K$. 
Observe that $\frac 1\alpha=t-\gamma=t-t_0+\varepsilon$ by the definition of $\gamma$.
By part 2) of Lemma~\ref{lemma:transferinformationfromalpha0toalpha},
\[
K_\nu=\mathcal{J}^\mu_{\underline r(\alpha),0}(K_\nu)=\Bigl\{ x\in K_\nu \mid 
\sum_{j=1}^\infty\mu\bigl(B(x,r)\cap K(j^{-\alpha})\bigr)=\infty 
\text{ for all } r>0\Bigr\}.
\]
Proposition~\ref{lemma:conditionforpositivehittingprobability} implies that  
$K_\nu \cap E_{\alpha}(\omega)\neq \emptyset$ almost surely. 
 Since $K_\nu\subseteq K$, we are done.
\end{proof}

The following lemma is \cite[Lemma 2.3]{JJKLSX2017}:

\begin{lemma}\label{lemma:JJKLSX2.3}
Let $(X,\rho)$ be a compact $d$-Ahlfors regular metric space. Then, for any $u>0$, there
exists a probability space $(\Xi,\Gamma,\widetilde{\Pb})$ and, for each $\xi\in\Xi$, a 
compact set $A(\xi)\subseteq X$ such that, for any analytic set $E\subseteq X$, 
\[
\widetilde{\Pb}(\{\xi\in\Xi\mid A(\xi)\cap E=\emptyset \})=1
\]
if $\dimh E<u$, and 
\[
\Vert \dimh (A(\xi)\cap E) \Vert_{L^{\infty}(\widetilde{\Pb})}=\dimh E-u
\]
if $\dimh E\geq u$.
\end{lemma}

\begin{proof}[Proof of Theorem \ref{thm:dimhofintersectionwithsublevelset}]
With the lemmas established in this section so far, we are able to reduce the proof to a
standard co-dimension argument. The proof is inspired by that of \cite[Lemma 3.4]{KPX2000}.

Recall that $F_\mu(t)=\dimh {D_\mu(t)}\le t$. Our aim is to show inequality
\eqref{eq-min-dim-Ealpha}, i.e., for 
$t-F_\mu(t)< \frac 1\alpha\leq t$, almost surely one has 
\[
\dimh (E_{\alpha}(\omega)\cap {D_\mu(t)})\geq \frac{1}{\alpha}+F_\mu(t)-t.
\]

Note first that if we are able to 
prove the claim for every  $\alpha$ such that $t-F_\mu(t)<\frac 1\alpha<t$, the case  
$\frac1\alpha=t$ follows by monotonicity of the sets.
Now fix $t-F_\mu(t)<\frac 1\alpha<t$. We will show a bit more than claimed, 
namely, that for any analytic set $F\subseteq D_\mu(t)$ with 
$\dimh F>t-\frac 1\alpha=:\gamma$, we have that
\begin{equation}\label{equation:intersectionwithanalyticsubset}
 \dimh(E_{\alpha}(\omega)\cap F)\geq \dimh F-\gamma
\end{equation}
almost surely. Fix an analytic set $F\subseteq D_\mu(t)$ with $\dimh F>\gamma$ and let 
$0<u<\dimh F-\gamma$. Applying Lemma~\ref{lemma:JJKLSX2.3} with $X$ equal to a compact cube
in $\R^d$ containing $\spt\mu$ yields the existence of a probability space 
$(\Xi,\Gamma,\widetilde{\Pb})$ and, for each $\xi\in \Xi$, a compact set $A(\xi)\subseteq X$
such that, for any analytic set $E\subseteq D_\mu(t)$, $\widetilde{\Pb}$-almost surely we 
have:
\begin{itemize}
\item [i)]
 if $\dimh E<u$, then $A(\xi)\cap E=\emptyset$.
 \item  [ii)]
 if $\dimh E\geq u$, then
\begin{equation}\label{equation:esssupofdimhofintersection}
\Vert \dimh(A(\xi)\cap E)\Vert_{L^{\infty}(\widetilde{\Pb})}=\dimh E-u.
\end{equation}
 \end{itemize}
We apply this to the analytic set $F$ (recall that $\dimh F>u$). For 
$\underline{\xi}:=(\xi_n)_{n=1}^\infty\in \Xi^{\N}$, let 
$\textbf{A}(\underline{\xi}):=\bigcup_{n=1}^\infty A(\xi_n)$. 
For every $\ell\geq 1$, set 
\[
p_\ell:=\widetilde\Pb\bigl(\{ \xi\in \Xi\mid \dimh (A(\xi)\cap F) \geq \dimh F-u-\tfrac{1}{\ell}\}\bigr).
\]
By \eqref{equation:esssupofdimhofintersection}, $p_\ell>0$, so for 
$\widetilde{\Pb}^{\N}$-almost every $\underline{\xi}$ there exists $\xi_n$ such that 
$\dimh (A(\xi_n) \cap F) \geq \dimh F-u-\frac{1}{\ell}$. Applying this for every integer 
$\ell$, we obtain that 
\[
\dimh({\textbf{A}}(\underline{\xi})\cap F)=\sup_n \dimh (A(\xi_n)\cap F)=\dimh F-u>\gamma
\]
almost surely with respect to the probability measure $\widetilde{\Pb}^{\N}$. Thus, for 
$\widetilde{\Pb}^{\N}$-typical $\underline{\xi}$, the analytic set 
$\textbf{A}(\underline{\xi})\cap F$ contains a compact subset included in $D_\mu(t)$ of 
dimension larger than $\gamma$ (see \cite{D1952,H1995}). Hence, by
Lemma~\ref{lemma:codimensionargumentunlocked}, 
\begin{equation}
\label{nonempty-intersect}
E_{\alpha}(\omega)\cap \textbf{A}(\underline{\xi})\cap F\neq \emptyset
\end{equation}
for $\Pb\times \widetilde{\Pb}^{\N}$-almost every $(\omega, \underline{\xi})$. Observe that, 
by the first item i) above applied to $E= E_{\alpha}(\omega)\cap F$,
for any $\omega \in \Omega$, if $\dimh (E_{\alpha}(\omega)\cap F)<u$, 
then $\widetilde{\Pb}$-almost surely $E_{\alpha}(\omega)\cap F\cap A(\xi)=\emptyset$. 
This, in turn, would imply that 
$E_{\alpha}(\omega)\cap F\cap\textbf{A}(\underline{\xi})=\emptyset$ for 
$\widetilde{\Pb}^{\N}$-almost every $\underline{\xi}$, contradicting 
\eqref{nonempty-intersect} if 
$\Pb(\{\omega\in\Omega\mid\dimh(E_{\alpha}(\omega)\cap F)<u\})>0$. Thus, we must have that 
\[
\dimh (E_{\alpha}(\omega)\cap F)\geq u
\]
$\Pb$-almost surely. Letting $u\nearrow \dimh F-\gamma$ through a countable sequence 
completes the proof.
\end{proof}

\section{Optimality of inequality \eqref{bounds-fmu}  in 
         Theorem~\ref{theorem:dimhatleastconjecture}  }\label{section:examples}

In this section, we construct two examples illustrating the possible dimensional behaviour of
random covering sets in the case where the Hausdorff and packing spectra of the generating
measure do not agree. In particular, we construct a measure $\muTwo$ in 
Section~\ref{example:dimensionispackingspectrum} showing that the 
Ekstr\"om--Persson conjecture is not true in the regime $\frac 1\alpha>\dimhu\muTwo$.

\subsection{Some technical lemmas}

\begin{lemma}\label{lemma:sumconvergestoexpectationas}
Let $(\widetilde\Omega,\widetilde\Pb)$ be a probability space. Let $(A_n)_{n=1}^{\infty}$ be 
a sequence of independent events in $\widetilde\Omega$, and let $(a_n)_{n=1}^{\infty}$ be a 
bounded sequence of positive numbers. Suppose that $(N_n)_{n=1}^{\infty}$ is a strictly
increasing sequence of natural numbers such that 
\[
\sum_{n=1}^\infty E_n^{-1}<\infty,
\]
where $E_n:=\sum_{k=N_{n-1}+1}^{N_n}\widetilde\Pb (A_k)a_k> 0$. For all 
$\tilde\omega\in\widetilde\Omega$, define 
\[
S_n(\tilde\omega):=\sum_{k=N_{n-1}+1}^{N_n} \chi_{A_k}(\tilde\omega) a_k.
\]
Then $\lim_{n\to\infty}\frac{S_n}{E_n}=1$ almost surely.
\end{lemma}

\begin{proof}
Let $Z_n:=\frac{S_n}{E_n}$. Then $\E(Z_n)=1$ and
\begin{align*}
    \mathrm{Var}(Z_n)&=E_n^{-2}\sum_{k=N_{n-1}+1}^{N_n}a_k^2\mathrm{Var}(\chi_{A_k})
    \leq E_n^{-2}\sum_{k=N_{n-1}+1}^{N_n}a_k^2 \E(\chi_{A_k})\\
    &=E_n^{-2}\sum_{k=N_{n-1}+1}^{N_n}a_k^2 \widetilde\Pb(A_k)
    \leq a_0E_n^{-1},
\end{align*}
where $a_0\coloneqq \sup_ka_k<\infty.$ Let $\varepsilon>0$. By the Chebyshev's inequality,
\[
\sum_{n=1}^\infty\widetilde\Pb\{ |Z_n-1|\geq\varepsilon\}\leq \varepsilon^{-2}
\sum_{n=1}^\infty\mathrm{Var}(Z_n)\leq a_0\varepsilon^{-2}\sum_{n=1}^\infty E_n^{-1}<\infty.
\]
By the Borel--Cantelli lemma, 
\[
\limsup_{n\to \infty} |Z_n-1|\leq \varepsilon
\]
almost surely. Thus, 
\[
\lim_{n\to\infty}\frac{S_n}{E_n}=\lim_{n\to\infty}Z_n=1
\]
almost surely.
\end{proof}

The following lemma is \cite[Lemma 9.5]{EP2018}:

\begin{lemma}\label{lemma:EP9.5}
Let $\underline{\mu}$ and $\underline{\nu}$ be sequences of probability measures on a 
separable metric space $X$ and let $\underline{r}$ be a sequence of positive numbers. Assume
that there is a constant $C$ such that, for $k=1,2,\ldots$ and for every $x\in \spt \mu_k$,
\[
\mu_k(B(x,2r_k))\leq C\nu_k(B(x,r_k)).
\]
Then $f_{\underline{\mu}}(\underline{r})\leq f_{\underline{\nu}}(2\underline{r})$.
\end{lemma}

We note that, using Lemma~\ref{lemma:EP9.5}, Ekstr\"om and Persson proved another lemma 
\cite[Lemma 9.6]{EP2018} which would be sufficient to show that Conjecture~\ref{conjecture:EP} 
is not valid for all $\alpha>0$ but, in our example, that lemma does not yield the correct 
lower bound for $f_{\mu}(\alpha)$ for all $\alpha>0$. To remedy this, we have the following
substitute for \cite[Lemma 9.6]{EP2018}, whose proof also uses Lemma~\ref{lemma:EP9.5}.

\begin{lemma}\label{lemma:modifiedEP9.6}
Let $\mu$ and $\theta$ be probability measures on a separable metric space $X$, let 
$(V_k)_{k=1}^{\infty}$ be a sequence of Borel subsets of $X$ such that $0<\mu(V_k)<1$ for
every $k\in\N$ and let $\underline r$ be a sequence of positive numbers tending to $0$.
Suppose that there exists a number $s>0$ such that 
\begin{equation}\label{s2assumption}
f_{\theta}(\underline{\rho})=s_2(\underline{\rho})
\end{equation}
for every sequence $\underline{\rho}=(\rho_k)_{k=1}^\infty$ of radii with 
$s_2(\underline{\rho})\leq s$. Assume also that
there exists a constant $C$ such that, for every $k\in\N$ and every $x\in \spt \theta$,
\[
\theta(B(x,2r_k))\leq C\mu_{V_k}(B(x,r_k))
\] 
where $\mu_{V_k}:=\mu(V_k)^{-1}\mu_{|_{V_k}}$.
If there exists $t\in\mathopen]0,s\mathclose[$ such that 
\[
\sum_{k=1}^{\infty}\mu(V_k)r_k^t=\infty,
\]
then $f_{\mu}(\underline{r})\geq t$.
\end{lemma}

\begin{proof}
The proof follows the same lines as that of \cite[Lemma 9.6]{EP2018}.

Write $p_k:=\mu(V_k)$, $\mu_k^0:=\mu_{V_k^c}$ and $\mu_k^1:=\mu_{V_k}$. For 
$\sigma\in \Sigma\coloneqq \{0,1 \}^{\N}$, let 
$\Pb_{\mu}^{\sigma}:=\prod_{k=1}^{\infty}\mu_k^{\sigma_k}$ on $\Omega=X^{\N}$ and define 
a probability measure $\Pb_{\underline{p}}$ on $\Sigma$ by setting 
\[
\Pb_{\underline{p}}:=\prod_{k=1}^{\infty}((1-p_k)\delta_0+p_k\delta_1).
\]
Then $\mu=(1-p_k)\mu_k^0+p_k\mu_k^1$ for every $k\in\N$. By \cite[(9.1)]{EP2018}, the 
probability measure $\Pb= \mu^{\N}$ on $\Omega$ can be decomposed into 
\[
\mu^{\N}=\int_\Sigma \Pb_{\mu}^{\sigma}\, d\Pb_{\underline{p}}(\sigma).
\]

Let $t\in\mathopen]0,s\mathclose[$ be such that $\sum_{k=1}^\infty p_kr_k^t=\infty$ and let 
$(N_n)_{n=0}^{\infty}$ be a strictly increasing sequence of natural numbers such that $N_0=0$
and
\[
E_n\coloneqq \sum_{k=N_{n-1}+1}^{N_n}p_kr_k^t\geq 2^n
\]
for every $n\geq 1$. Define measurable sets 
\[
G:=\{ \omega\in \Omega\mid \dimh E_{\underline{r}}(\omega)\geq t \}
\]
and
\[
A:=\Bigl\{ \sigma \in \Sigma\mid \text{ there exists }n_0\in \N \text{ such that } 
\sum_{k=N_{n-1}+1}^{N_n} \sigma_k r_k^t\geq\frac{E_n}2 \text{ for all }n\geq n_0 \Bigr\}.
\]
For $\sigma \in \Sigma$ and $\omega\in \Omega$, let 
\[
E_{\underline{r}}(\sigma,\omega):=\limsup_{\sigma_k=1} B(\omega_k,r_k)
\]
and, for $\sigma\in \Sigma$, define
\[
G_{\sigma}:=\{\omega\in \Omega\mid \dimh E_{\underline{r}}(\sigma,\omega)\geq t \}.
\]
Then $G_{\sigma}\subseteq G$ for every $\sigma\in\Sigma$, hence
\[
\mu^{\N}(G)=\int \Pb_{\mu}^{\sigma}(G)\, d\Pb_{\underline{p}}(\sigma)\geq \int_A 
\Pb_{\mu}^{\sigma}(G_{\sigma})\, d\Pb_{\underline{p}}(\sigma).
\]
Since $\sum_{n=1}^\infty E_n^{-1}<\infty$, we conclude, by applying 
Lemma~\ref{lemma:sumconvergestoexpectationas} with 
$(\widetilde\Omega,\widetilde\Pb)=(\Sigma,\Pb_{\underline p})$ and 
$A_n=\{\sigma\in\Sigma\mid \sigma_n=1\}$, that $\Pb_{\underline{p}}(A)=1$. Thus, to conclude 
that $\mu^{\N}(G)=1$, which implies that $f_\mu(\underline r)\geq t$, it suffices
to show that $\Pb_{\mu}^{\sigma}(G_{\sigma})=1$ for $\sigma\in A$.

Let $\sigma \in A$ and let $(\ell_n)_{n=1}^\infty$ be the strictly increasing enumeration of 
$\{\ell\mid \sigma_{\ell}=1 \}$. Define a sequence $\underline{\rho}$ by setting 
$\rho_n:=\frac 12r_{\ell_n}$ for all $n\in\N$. Then, since $\lim_{n\to\infty}E_n=\infty$ and 
\[
\sum_{k=N_{n-1}+1}^{N_n} \sigma_k r_k^t\geq \frac{E_n}2
\]
for all large $n\in\N$, we have, for large $n\in\N$, that 
\[
\sum_{k=1}^\infty \rho_k^t=\sum_{k=1}^\infty r_{\ell_k}^t
\geq \sum_{N_{n-1}+1\leq \ell_k\leq N_n}r_{\ell_k}^t
=\sum_{j=N_{n-1}+1}^{N_n} \sigma_jr_j^t\geq \frac{E_n}2\xrightarrow[n\to\infty]{} \infty.
\]
Hence, $\sum_{k=1}^\infty \rho_k^t=\infty$, implying that $s_2(\underline\rho)\ge t$. If 
$s_2(\underline\rho)\le s$, we have that 
$f_{\theta}(\underline{\rho})=s_2(\underline{\rho})$ by \eqref{s2assumption}. If
$s_2(\underline\rho)>s$, consider the sequence 
$\underline{\tilde\rho}=(\tilde\rho_n)_{n=1}^\infty$ where 
\[
\tilde\rho_n:=\rho_n^{\frac{s_2(\underline\rho)}s}<\rho_n.
\]
Now $s_2(\underline{\tilde\rho}) =s$, so 
$f_\theta(\underline\rho)\ge f_\theta(\underline{\tilde\rho})=s$ by \eqref{s2assumption}. 
Thus,
\[
f_{\theta} (\underline{\rho})\geq \min \{s, s_2(\underline{\rho}) \}\geq t.
\]
Define $\pi\colon \Omega \to \Omega$ by setting $\pi(\omega)_n:=\omega_{\ell_n}$. Then
\[
\pi_\ast\Pb_{\mu}^{\sigma}=\prod_{n=1}^{\infty}\mu_{\ell_n}^{1}
\]
and
\[
E_{2\underline{\rho}}(\pi(\omega))=E_{\underline{r}}(\sigma,\omega).
\]
Applying Lemma~\ref{lemma:EP9.5} with $\mu_k=\theta$ and $\nu_k=\mu_{\ell_k}^1$ for all
$k\in\N$, we have for $\Pb_{\mu}^{\sigma}$-almost every $\omega\in \Omega$ that
\[
\dimh E_{\underline{r}}(\sigma,\omega)= \dimh E_{2\underline{\rho}}(\pi(\omega))
=f_{(\mu_{\ell_n}^1)_{n=1}^{\infty}}(2\underline{\rho})\geq f_{\theta}(\underline{\rho})
\geq t.
\]
Thus $\Pb_{\mu}^{\sigma}(G_{\sigma})=1$.
\end{proof}

\subsection{A measure $\muOne$ which realises the lower bound in 
            \eqref{bounds-fmu}}\label{example:dimensionishausdorffspectrum}

The following example shows that it is possible that 
$f_{\mu}(\alpha)=\overline{F}_{\mu}(\frac 1\alpha)<\overline{H}_{\mu}(\frac 1\alpha)$,
implying the sharpness of the lower bound in 
Theorem~\ref{theorem:dimhatleastconjecture}.

Fix $0<s<u<1$ and let $0<a<b<\frac{1}{2}$ be such that
$s=\frac{\log 2}{-\log a}$ and $u=\frac{\log 2}{-\log b}$, i.e., $2=a^{-s} = b^{-u}$. Let
$(N_k)_{k=1}^\infty$ be a rapidly growing sequence of integers and let
$C:=\bigcap_{n=0}^{\infty}C_n$ be the Cantor set, where $C_{n}$ is obtained from
$C_{n-1}$ by removing the middle $(1-2a)$-interval from each
construction interval of $C_{n-1}$ for $N_{2k}\leq n< N_{2k+1}$ and by removing
the middle $(1-2b)$-interval from each construction interval of
$C_{n-1}$ for $N_{2k+1}\leq n<N_{2k+2}$. If the
sequence $(N_k)_{k=1}^\infty$ grows fast enough, then $\dimh C=s$ and $\dimp C=u$. 

For each $k=0,1,2,\ldots$, let $\mu_k$ be the sum of $2^k$ Dirac masses located at the
midpoints of the gaps which appear when $C_{k+1}$ is constructed from $C_k$. Let $\beta>1$ 
and define a measure $\muOne$ by setting
\[
\muOne:=\frac{\sum_{k=0}^{\infty}2^{-\beta k}\mu_k}{\sum_{k=0}^{\infty}2^{(1-\beta)k}}
=(1-2^{1-\beta})\sum_{k=0}^{\infty}2^{-\beta k}\mu_k.
\]
Note that $\spt \muOne = C\cup\bigl(\bigcup_{k=0}^{\infty}\spt \mu_k\bigr)$ and 
$\dimhu (\muOne)=0$. Clearly, if $x\not \in \spt \muOne$,
then $\dimloc\muOne(x)=\infty$. We then investigate the local dimensions of 
$\muOne$ on its support.

\begin{lemma}\label{dim-local-xi}
If $x\in \spt \mu_k$ for some $k\in\N$, then $\dimloc\muOne(x)=0$. For every
$x\in C$, $\dimlocl \muOne(x)=\beta s$ and $\dimlocu \muOne(x)=\beta u$.  
\end{lemma}

\begin{proof}
The first claim is immediate. Let $\ell_n$ denote the length of the level $n$
construction intervals of $C$. Then $a^n\leq \ell_n \leq b^n$ and 
$a\ell_n\leq \ell_{n+1}\leq b \ell_n$ for every $n\in\N$. Also, since the sequence 
$(N_n)_{n=1}^\infty$ is rapidly growing, for every $\varepsilon >0$, 
there are infinitely many $n$ such that $a^n\leq \ell_n  \leq a^{n(1-\ep)}$ 
and there are infinitely many $n$ such that $b^{n(1+\ep)} \leq \ell_n\leq b^n$.

For $x\in C$, denote by $C_n(x)$ the construction 
interval at level $n$ containing $x$. If $x\in C$ and $r\in\mathopen]0,1\mathclose[$ is such 
that $\ell_{n+1}\leq r< \ell_n$, then $C_{n+1}(x)\subseteq B(x,r)$ and 
$B(x,r)\cap \spt \mu_k=\emptyset$ for $k\leq n-m_0$, where $b^{m_0}<\frac 12(1-2b)$. Thus, 
$\mu_{n+m}(B(x,r))=0$ for $m\leq -m_0$ and $\mu_{n+m}(B(x,r))\approx 2^m$ for $m\geq 1$.
This implies that (recalling that $\beta>1$) 
\[
\muOne(B(x,r))\approx \sum_{m=0}^{\infty}2^{-\beta(n+m)+m}\approx 2^{-\beta n}.
\]
Since $r\approx\ell_n$ and $\ell_n$ can be of the same order of either $a^n$ or $b^n$ for 
suitable integers $n$, we  conclude that $\dimlocl \muOne(x)=\beta s$ and 
$\dimlocu \muOne(x)=\beta u$. 
\end{proof}

Given that $\dimh (\spt \mu_k) =0$, $\dimh C=s$ and $\dimp C=u$, Lemma~\ref{dim-local-xi} 
directly gives  
\begin{align*}
    \overline{F}_{\muOne}(t)=\begin{cases}
        0& \text{ if }t<\beta s-s, \\
        t-(\beta s-s)& \text{ if } \beta s-s\leq t < \beta s,\\
        s&  \text{ if } \beta s \leq t.
    \end{cases}
\end{align*}
Also, if $\beta s\le u$, we have that 
\begin{align*}
    \overline{H}_{\muOne}(t)=\begin{cases}
        t-(\beta s-u)& \text{ if }0\leq t<\beta s, \\
        u& \text{ if } \beta s\leq t,
    \end{cases}
\end{align*}
and if $u<\beta s$, we obtain that
\begin{align*}
    \overline{H}_{\muOne}(t)=\begin{cases}
        0& \text{ if }0\le t<\beta s-u, \\
        t-(\beta s-u)& \text{ if } \beta s-u\leq t < \beta s,\\
        u&  \text{ if } \beta s \leq t.
    \end{cases}
\end{align*}

Figure~\ref{figure:lipschitzhullofspectrums} shows the graphs of $\overline{F}_{\muOne}$ and 
$\overline{H}_{\muOne}$ in the case $1<\beta<\frac us$. In this case, 
$\overline{H}_{\muOne}(\frac 1\alpha)$ is a worse upper bound for $f_{\muOne}(\alpha)$ than 
the trivial upper bound $\frac 1\alpha$.

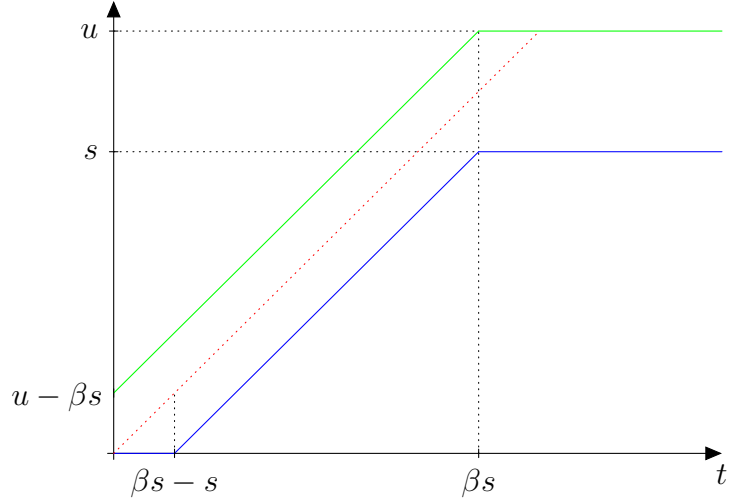
\begin{figure}[ht!]
\begin{center}
        \begin{tikzpicture}[line cap=round, line join=round, >=triangle 45,
                            x=1.0cm, y=1.0cm, scale=0.8]

          \draw [->,color=black] (-0.1,0) -- (10,0);
          \draw [->,color=black] (0,-0.1) -- (0,7.5);
          \draw [color=blue] (0,0) -- (1,0);
          \draw [color=blue] (1,0) -- (6,5);
          \draw [color=blue] (6,5) -- (10,5);

           \draw [color=green] (0,1) -- (6,7);
          \draw [color=green] (6,7) -- (10,7);

          \draw (10,0) node [below] {$t$};

          \draw[shift={(1,0)}] (0pt, 2pt) -- (0pt,-2pt) node [below] {$\beta s-s$};
           \draw[shift={(0,1)}] (0pt, 2pt) -- (0pt,-2pt) node [left] {$u-\beta s$};
          \draw[shift={(6,0)}] (0pt, 2pt) -- (0pt,-2pt) node [below] {$\beta s$};
         
          \draw[shift={(0,5)}] (2pt, 0pt) -- (-2pt,0pt) node [left] {$s$};
          \draw[shift={(0,7)}] (2pt, 0pt) -- (-2pt,0pt) node [left] {$u$};
          
          \draw[dotted] (1,0) -- (1,1);
          \draw[dotted] (0,5) -- (6, 5);
          \draw[dotted] (6,0) -- (6,7);
          \draw[dotted] (0,7) -- (6,7);
           \draw[dotted][color=red] (0,0) -- (7,7);
        \end{tikzpicture}
\end{center}
\caption{Graph of the functions $\overline{F}_{\muOne}$ and $\overline{H}_{\muOne}$ highlighted in blue and green, respectively. The diagonal is depicted in red.}
\label{figure:lipschitzhullofspectrums}
\end{figure}

Note that since $\dimhu({\muOne})=0$, the equality 
$f_{\muOne}(\alpha)=\overline{F}_{\muOne}(\frac 1\alpha)= \frac{1}{\alpha}$ never occurs. 

\begin{lemma}
For every $\alpha$ with $ \frac1\alpha<\beta s$, we have that
$f_{\muOne}(\alpha)= \overline{F}_{\muOne}(\frac 1\alpha)$.
\end{lemma}

\begin{proof}
Let $\alpha>0$ be such that $\frac 1\alpha <\beta s$. By 
Theorem~\ref{theorem:dimhatleastconjecture}, it is enough to prove that 
$f_{\muOne}(\alpha) \leq \overline{F}_{\muOne}(\frac 1\alpha)$.
Note that 
\[
\limsup_{k\to\infty}B(\omega_k,k^{-\alpha} )=\Bigl(\limsup_{\omega_k\in C(k^{-\alpha} )}
B(\omega_k,k^{-\alpha} )\Bigr)
\bigcup\Bigl(\limsup_{\omega_k\not\in C(k^{-\alpha} )}B(\omega_k,k^{-\alpha} )\Bigr).
\]
Since $C\cap B(x,r)=\emptyset$ if $x\not \in C(r)$, we have that
\[
\limsup_{\omega_k\not\in C(k^{-\alpha} )}B(\omega_k,k^{-\alpha} )\subseteq
\bigcup_{j=0}^\infty\spt\mu_j
\]
which is countable. Thus, 
\[
\dimh E_{\alpha}(\omega)=\dimh\Bigl( \limsup_{\omega_k\in C(k^{-\alpha} )}
B(\omega_k,k^{-\alpha} )\Bigr)
\]
for every $\omega\in\Omega$.  
 
Let us prove that, for every $t>\frac 1\alpha+s-\beta s$, 
\[
\mc{H}_{\infty}^t\Bigl(\limsup_{\omega_k\in C(k^{-\alpha} )}B(\omega_k,k^{-\alpha} )\Bigr)
=0.
\]

For this, observe that, if $r\in\mathopen]0,1\mathclose[$ is such that 
$\ell_{n+1}\leq r<\ell_n$ then 
\[
\muOne(C(r))\approx 2^{-\beta n}2^n=\ell_n^{n(\beta-1)\frac{\log 2}{-\log \ell_n}}\approx 
r^{n(\beta-1)\frac{\log 2}{-\log \ell_n}}.
\]
Since $\ell_n\ge a^n$ for all $n\in\N$, we have that $\muOne(C(r))\lesssim r^{\beta s-s}$ for 
every $0<r<1$. Thus, if $\frac 1\alpha<\beta s-s$, then 
\[
\sum_{k=1}^\infty \Pb(\{ \omega_k \in C(k^{-\alpha} ) \})=\sum_{k=1}^\infty 
\muOne(C(k^{-\alpha} ))\lesssim \sum_{k=1}^\infty k^{-\alpha(\beta s-s)}< \infty,
\]
whence, by the Borel--Cantelli lemma,
\[
\Pb(\{ \omega_k \in C(k^{-\alpha} ) \text{ for infinitely many } k \})=0.
\]
This implies that $C\cap E_{\alpha}(\omega)=\emptyset$ almost surely and, thus, 
$\dimh E_{\alpha}(\omega)=0$ almost surely. Furthermore, if $\beta s-s<\frac 1\alpha<\beta s$ 
and $t>\frac 1\alpha+s-\beta s$, then  
\begin{align*}
\E\Bigl(\sum_{k=1}^\infty\mc{H}_{\infty}^t\bigl(\chi_{C(k^{-\alpha} )}
(\omega_k)B(\omega_k,k^{-\alpha} )\bigr)\Bigr) &\leq 2^t\sum_{k=1}^\infty 
\muOne(C(k^{-\alpha} ))(k^{-\alpha} )^t\\
&\lesssim \sum_{k=1}^\infty k^{-\alpha(t+\beta s-s)}
<\infty.
\end{align*}
Thus, 
$\sum_{k=1}^\infty\mc{H}_{\infty}^t\bigl(\chi_{C(k^{-\alpha} )}(\omega_k)
B(\omega_k,k^{-\alpha} )\bigr)<\infty$ 
almost surely, implying that, almost surely for any $n\in\N$,
\begin{align*}
\mc{H}_{\infty}^t\Bigl(\limsup_{\omega_k\in C(k^{-\alpha} )}B(\omega_k,k^{-\alpha} )\Bigr)\leq 
\sum_{{k=n}}^\infty\mc{H}_{\infty}^t\bigl(\chi_{C(k^{-\alpha} )}(\omega_k)
B(\omega_k,k^{-\alpha} )\bigr){\xrightarrow[n\to\infty]{}0}.
\end{align*}
Therefore, 
$\dimh\bigl(\limsup_{\omega_k\in C(k^{-\alpha} )}B(\omega_k,k^{-\alpha} )\bigr)\leq t$ almost 
surely. This holds true for every $t>\frac 1\alpha+s-\beta s$, and we 
conclude that 
$f_{\muOne}(\alpha)\leq\frac 1\alpha +s-\beta s=\overline{F}_{\muOne}(\frac 1\alpha)$.
\end{proof}

\subsection{A measure $\muTwo$ which realises the upper bound in \eqref{bounds-fmu}}
\label{example:dimensionispackingspectrum}

The following example demonstrates that it is possible that 
$f_{\mu}(\alpha)=\overline{H}_{\mu}(\frac 1\alpha)>\overline{F}_{\mu}(\frac 1\alpha)$
implying that strict inequality in Theorem~\ref{theorem:dimhatleastconjecture} is possible,
hence showing that Conjecture~\ref{conjecture:EP} is not always true.
 
Let $s$, $u$, $a$, $b$, $C=\bigcap_{n=0}^\infty C_n$, $\theta$ and $\mu_k$ for $k=0,1,\dots$ 
be as in Section~\ref{example:dimensionishausdorffspectrum}. Let $\theta$ be the natural 
Borel probability measure on $C$ giving equal weight to all construction
intervals at same level. We then have that 
\begin{equation}\label{equation:estimatesforcantormeasure}
r^{u}\lesssim \theta(B(x,r)) \lesssim r^{s}    
\end{equation}
for all $x\in C$ and $0<r<1$, where the notation $g(r)\lesssim h(r)$
means that there exists a constant $D$ such that $g(r)\le Dh(r)$. If the
sequence $(N_k)_{k=1}^\infty$ grows fast enough, then $\dimlocl \theta(x)=s$ and
$\dimlocu \theta(x)=u$ for all $x\in C$.  

Let 
\[
\beta>\frac us>1
\]
and define a measure $\muTwo$ by setting
\[
\muTwo:=\gamma_0\sum_{k=0}^{\infty}2^{\beta \frac{\log \ell_k}{-\log a}}\mu_k,
\]
where
\[
\gamma_0\coloneqq \left( \sum_{k=0}^{\infty} 2^{\beta \frac{\log \ell_k}{-\log a}+k} 
\right)^{-1}.
\]
Observe first that since $\frac{\log b}{\log a}= \frac{s}{u}>\frac{1}{\beta}$ and
\[
k\log a\leq \log \ell_k\leq k\log b,
\]
we obtain that
\[
\sum_{k=0}^{\infty} 2^{\beta \frac{\log \ell_k}{-\log a}+k}\leq \sum_{k=0}^{\infty}
2^{k(1-\beta \frac{s}{u})}=:D_0<\infty.
\]
Thus, $\gamma_0 \in\mathopen]0,\infty\mathclose[$.
Note that $\spt \muTwo = C\cup (\bigcup_{k=0}^{\infty}\spt \mu_k)$.

As for the measure $\muOne$, if $x\not \in \spt \muTwo$ then  
$\dimloc\muTwo(x)=\infty$, if $x\in \spt \mu_k$ for some $k\in\N$ then 
$\dimloc\muTwo(x)=0$ and, finally, $\dimhu\muTwo =0$.
The difference with $\muOne$ is that the local dimensions of $\muTwo$ exist on $C$.

\begin{lemma}
For every $x\in C$, we have that $\dimlocl \muTwo(x)=\beta s$.
\end{lemma}

\begin{proof}
Let us first  show that, for any $n\in\N$, 
\begin{equation}\label{equation:allweightscomparabletolargest}
\sum_{m=0}^{\infty}2^{-\beta \frac{\log \ell_{n+m}}{\log a}+m}
\approx 2^{-\beta \frac{\log \ell_n}{\log a}}.
\end{equation}
The inequality $"\geq"$ in \eqref{equation:allweightscomparabletolargest} is clear. To see 
the inequality $"\lesssim"$, observe that $\ell_{n+m}\leq \ell_n b^m$ for all $n,m\in\N$,
hence 
\[
\frac{\log \ell_{n+m}}{\log a}\geq \frac{\log\ell_n+m\log b}{\log a}
=\frac{\log \ell_{n}}{\log a}+m\frac{s}{u}.
\]
This yields
\[
\sum_{m=0}^{\infty}2^{-\beta \frac{\log \ell_{n+m}}{\log a}+m}
\leq 2^{-\beta \frac{\log \ell_{n}}{\log a}}\sum_{m=0}^{\infty}2^{m(1-\beta \frac{s}{u})}
=D_02^{-\beta \frac{\log \ell_{n}}{\log a}}.
 \]
 
Suppose that $x\in C$ and $r\in\mathopen]0,1\mathclose[$ is such that 
$\ell_{n+1}\leq r< \ell_n$. Then $C_{n+1}(x)\subseteq B(x,r)$, giving 
$\mu_{n+1}(B(x,r))\geq 1$ and 
\[
\muTwo(B(x,r))\gtrsim 2^{\beta \frac{\log \ell_{n+1}}{-\log a}}=\ell_{n+1}^{\beta s} 
\approx r^{\beta s}.
\]
On the other hand, as in Section~\ref{example:dimensionishausdorffspectrum}, there is 
$m_0\in\N$ such that $B(x,r)\cap \spt \mu_k=\emptyset$ for $k\leq n-m_0$. This implies that
$B(x,r)\cap \spt \muTwo \subseteq C_{n-m_0+1}(x)$ and, thus, for $m=0,1,2,\ldots$,
\[
\mu_{n-m_0+1+m}(B(x,r))\leq 2^m.
\]
Hence, by \eqref{equation:allweightscomparabletolargest},
\[
\muTwo(B(x,r))\lesssim \sum_{m=0}^{\infty} 2^{-\beta \frac{\log \ell_{n-m_0+1+m}}{\log a}+m}
\lesssim 2^{-\beta \frac{\log \ell_{n-m_0+1}}{\log a}}
\lesssim_{m_0}2^{-\beta \frac{\log \ell_{n}}{\log a}}=\ell_n^{\beta s}\approx r^{\beta s}.
\]
This proves the lemma.
\end{proof}

The previous lemma  yields
\begin{align*}
    \overline{F}_{\muTwo}(t)=\begin{cases}
        0& \text{ if }t<\beta s-s, \\
        t-(\beta s-s)& \text{ if } \beta s-s\leq t < \beta s,\\
        s&  \text{ if } \beta s \leq t,
    \end{cases}
\end{align*}
and 
\begin{equation}\label{overH}
    \overline{H}_{\muTwo}(t)=\begin{cases}
        0&  \text{ if }0\leq t<\beta s-u,\\
        t-(\beta s-u)& \text{ if }\beta s-u\leq t<\beta s, \\
        u& \text{ if } \beta s\leq t.
    \end{cases}
\end{equation}

Figure~\ref{figure:dimensionandlipschitzhullofspectrums} shows the graphs of 
$f_{\muTwo}(\alpha)$ ({purple}), $\overline{F}_{\muTwo}(\frac 1\alpha)$ (blue) and 
$\overline{H}_{\muTwo}(\frac 1\alpha)$ (green) as a function of $\frac 1\alpha$. One sees in 
particular that $\overline{F}_{\muTwo}(\frac 1\alpha)<\overline{H}_{\muTwo}(\frac 1\alpha)$ 
whenever $\alpha>0$ is such that $\overline{H}_{\muTwo}(\frac 1\alpha)>0$.

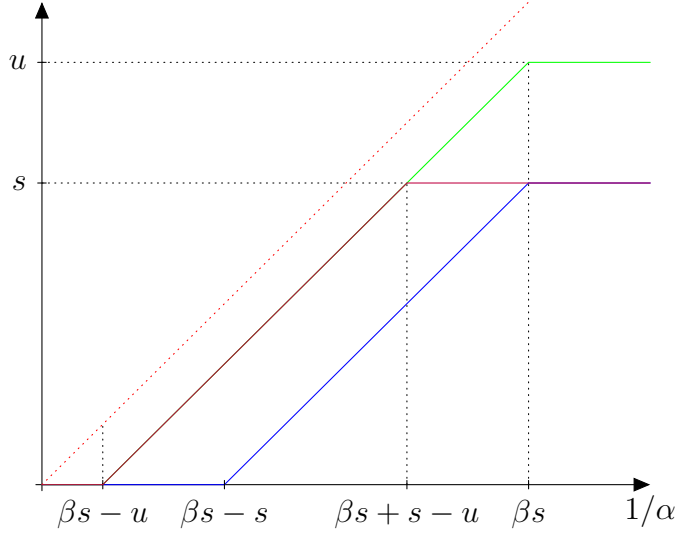
\begin{figure}[ht!]
\begin{center}
        \begin{tikzpicture}[line cap=round, line join=round, >=triangle 45,
                            x=1.0cm, y=1.0cm, scale=0.8]

          \draw [->,color=black] (-0.1,0) -- (10,0);
          \draw [->,color=black] (0,-0.1) -- (0,8);
          \draw [color=blue] (0,0) -- (3,0);
          \draw [color=blue] (3,0) -- (8,5);
          \draw [color=blue] (8,5) -- (10,5);

            \draw [color=green] (0,0) -- (1,0);
           \draw [color=green] (1,0) -- (8,7);
          \draw [color=green] (8,7) -- (10,7);

          \draw (10,0) node [below] {$1/\alpha$};

          \draw[shift={(3,0)}] (0pt, 2pt) -- (0pt,-2pt) node [below] {$\beta s-s$};
           \draw[shift={(1,0)}] (0pt, 2pt) -- (0pt,-2pt) node [below] {$\beta s-u$};
          \draw[shift={(8,0)}] (0pt, 2pt) -- (0pt,-2pt) node [below] {$\beta s$};
          \draw[shift={(6,0)}] (0pt, 2pt) -- (0pt,-2pt) node [below] {$\beta s+s-u$};
         
          \draw[shift={(0,5)}] (2pt, 0pt) -- (-2pt,0pt) node [left] {$s$};
          \draw[shift={(0,7)}] (2pt, 0pt) -- (-2pt,0pt) node [left] {$u$};
          
          \draw[dotted] (1,0) -- (1,1);
          \draw[dotted] (0,5) -- (6, 5);
          \draw[dotted] (8,0) -- (8,7);
         \draw[dotted] (6,0) -- (6,5);
          \draw[dotted] (0,7) -- (8,7);
           \draw[dotted][color=red] (0,0) -- (8,8);

           \draw[color=purple] (0,0) -- (1,0);
           \draw[color=purple] (1,0) -- (6,5);
           \draw[color=purple] (6,5) -- (10,5);
        \end{tikzpicture}
\end{center}
\caption{Graph of the functions $f_{\muTwo}(\alpha)$ (purple), 
  $\overline{F}_{\muTwo}(\frac 1\alpha)$ (blue) and $\overline{H}_{\muTwo}(\frac 1\alpha)$ 
  (green). The diagonal is depicted in red.}
\label{figure:dimensionandlipschitzhullofspectrums}
\end{figure}

\begin{lemma}
For every $\alpha$ such that $\beta s-u <\frac1\alpha$, we have that
$f_{\muTwo}(\alpha)=\min \{s, \overline{H}_{\muTwo}(\tfrac 1\alpha)\}$.
\end{lemma}

\begin{proof}
The upper bound $f_{\muTwo}(\alpha)\leq \min \{s, \overline{H}_{\muTwo}(\frac 1\alpha) \}$
follows from Theorem~\ref{theorem:EPThm2} and the fact that 
$E_\alpha(\omega)\subset C\cup\bigcup_{k=0}^\infty\spt\mu_k$. To prove the lower bound
$f_{\muTwo}(\alpha)\ge \min \{s, \overline{H}_{\muTwo}(\frac 1\alpha) \}$, we will apply 
Lemma~\ref{lemma:modifiedEP9.6}. To that end, let $\alpha>0$ be such that 
$\beta s-u<\frac 1\alpha<\beta s+s-u$  (the situation $\frac1\alpha \geq \beta s+s-u$ is 
immediate, see Figure~\ref{figure:dimensionandlipschitzhullofspectrums}).

Recall that $\overline{H}_{\muTwo}(\frac 1\alpha)=s$ is 
equivalent to $\frac 1\alpha=\beta s+s-u$, see \eqref{overH}. Set   
$V_k:=C(k^{-\alpha})$. If $0<r<1$ and $n\in\N$ is such that $\ell_{n+1}\leq r<\ell_n$, then 
\[
\muTwo(C(r))\approx 2^{-\beta \frac{\log \ell_n}{\log a}}2^n=\ell_n^{\beta s} 2^n
\approx r^{\beta s} 2^n\in [r^{\beta s-s}, r^{\beta s-u}].
\]
This rewrites 
\begin{equation}\label{inter-xitilde}
\muTwo(V_k)\ \in [k^{-\alpha(\beta s-s{)}}, k^{-\alpha(\beta s-u)}].
\end{equation}
Also, the bounds of the interval above are reached, in the sense that, since the sequence 
$(N_n)_{n=1}^\infty$ is rapidly growing, for every $\ep>0$, there
are infinitely many integers $k$ such that
$\muTwo(V_k){<} k^{-\alpha(\beta s-s{-\ep} )}$ 
and there are infinitely many integers $k$ such that
$\muTwo(V_k){>} k^{-\alpha(\beta s-u{+\ep}) }$. 

On the other hand, for $x\in C$ and $\ell_{n+1}\leq r<\ell_n$, recalling that $\theta$ is the 
uniform measure on $C$, we have that
\[
\theta (B(x,2r))\approx 2^{-n}\approx \muTwo(C(r))^{-1}r^{\beta s}
\approx \frac{\muTwo(C(r)\cap B(x,r))}{\muTwo(C(r))}.
\]
In \cite[Example 7.1]{JJMS2025}, it is calculated that 
$f_{\theta}(\underline{\rho})=\min\{s, s_2(\underline{\rho}) \}$ for any sequence 
$\underline{\rho}$. In order to apply Lemma~\ref{lemma:modifiedEP9.6} and obtain the lower
bound $f_{\muTwo}(\alpha)\geq \min \{s, \overline{H}_{\muTwo}(\frac 1\alpha)\}$, it suffices 
to show that 
\[
\sum_{k=1}^{\infty}\muTwo(V_k)(k^{-\alpha})^t=\infty
\]
for all $0<t<\frac 1\alpha-(\beta s-u)$ (see \eqref{overH}). Fix $\varepsilon>0$. Note 
that by \eqref{inter-xitilde} and the  remark following it, there exists a strictly 
increasing sequence $(k_i)_{i=1}^\infty$ of natural numbers such
that $\muTwo(V_{k_i})\geq (k_i^{-\alpha})^{\beta s-(u-\varepsilon)}$ for every $i\in \N$. 
Furthermore, if $\frac{k_i}2\leq j\leq k_i$, then 
\[
\muTwo(V_j)\geq \muTwo(V_{k_i}) \geq (k_i^{-\alpha})^{\beta s-(u-\varepsilon)}
\geq 2^{-\alpha (\beta s-(u-\varepsilon))}(j^{-\alpha})^{\beta s-(u-\varepsilon)}.
\]
Thus,
\begin{align*}
\sum_{k=1}^{\infty}\muTwo(V_k)(k^{-\alpha})^t
   &\geq \lim_{i\to \infty}\sum_{j={\left\lceil\frac{k_i}2\right\rceil}}^{k_i}2^{-\alpha (\beta s-(u-\varepsilon))}
     j^{-\alpha(\beta s-(u-\varepsilon))}j^{-\alpha t}\\
   &\gtrsim \lim_{i\to \infty} k_i^{1-\alpha(t+\beta s-u+\varepsilon)}=\infty
\end{align*}
provided $t+\beta s-u+\varepsilon<\frac 1\alpha$, that is, if 
$t<\frac 1\alpha-\beta s+u-\varepsilon$. By Lemma~\ref{lemma:modifiedEP9.6}, 
$f_{\muTwo}(\alpha)\geq\frac 1\alpha -\beta s+u-\varepsilon$ and letting $\varepsilon$ tend to
$0$ along a sequence yields 
\[
f_{\muTwo}(\alpha)\geq\frac 1\alpha -\beta s+u=\overline{H}_{\muTwo}(\tfrac 1\alpha),
\]
since $\beta s-u<\frac 1\alpha<\beta s+s-u$ and 
$\overline{H}_{\muTwo}(\frac 1\alpha)=\frac 1\alpha-(\beta s-u)$ in this range.
\end{proof}

\section{Application to hitting probabilities}
 \label{sec-hitting}
 
In \cite{JJKLSX2017}, the authors studied the hitting probability problem for analytic sets in 
Ahlfors regular metric spaces. They obtained the following result:

\begin{theorem}[Järvenpää, Järvenpää, Koivusalo, Li, Suomala, Xiao]
      \label{theorem:JJKLSX_hittingprob} 
Let $X$ be an Ahlfors regular metric space, $\mu$ the natural Ahlfors regular probability 
measure on $X$ and $\underline{r}$ a sequence of radii such that 
$s_2(\underline{r})\leq \dim X$. Then, for all analytic sets $F\subseteq X$, 
\begin{enumerate}
    \item[i)] if $\dimp F< \dim X-s_2(\underline{r})$, then $F\cap E_{\underline{r}}(\omega)
      =\emptyset$ almost surely,

    \item[ii)] if $\dimh F>\dim X-s_2(\underline{r})$, then $F\cap E_{\underline{r}}(\omega)
      \neq\emptyset$ almost surely and

    \item[iii)] if   $\alpha>0$  and $\dimp F>\dim X-\frac 1\alpha$, then 
      $F\cap E_{\alpha}(\omega)\neq\emptyset$ almost surely.
\end{enumerate}

\end{theorem}

In \cite{LS2014}, it is shown that for arbitrary sequences $\underline{r}$, the condition 
$\dimh F>\dim X-s_2(\underline{r})$ in the previous theorem cannot be relaxed by replacing 
the Hausdorff dimension with the packing dimension, even in the case of the torus. Using the 
tools developed in  Section~\ref{section:hittingprob}, we give below a new proof for the 
above theorem in $\R^d$. 

Let us emphasise that, as claimed in the introduction, a somewhat interesting consequence of 
this proof is that, at least in the case of compact sets, the condition in item ii) in
Theorem~\ref{theorem:JJKLSX_hittingprob} for the Hausdorff
dimension of the intersecting set can be replaced by the modified lower box counting
dimension, defined for every set $E$ by
\[
\dimMBl E =\inf\Bigl\{\sup_{i\in I}{ \dimBl F_i}\mid E\subset \bigcup_{i\in I}F_i\Bigr\},
\]
where $\dimBl$ is the lower box counting dimension. Similarly, the modified upper 
box counting dimension is
\begin{equation}\label{defdimMBu}
\dimMBu E =\inf\Bigl\{\sup_{i\in I}{ \dimBu F_i}\mid E\subset \bigcup_{i\in I}F_i\Bigr\}.
\end{equation}
It is well known that $\dimMBu E=\dimp E$ and $\dimh E\le\dimMBl E$ for all sets 
$E$, and the strict inequality $\dimh E<\dimMBl E$ is possible (see e.g. 
\cite{Fal2003,Mat1995}). 

\begin{proof}[Proof of Theorem~\ref{theorem:hittingprobforregularmeasure}]
Write $s_2:=s_2(\underline{r})$.  We start by proving the theorem for compact sets.

Let us first prove i). Note that 
$C\cap E_{\underline{r}}(\omega)=\emptyset$ is a tail event. Suppose that 
$C\cap E_{\underline{r}}(\omega)\neq \emptyset$ almost surely. We will show that this implies 
that $\dimp C\geq d-s_2$. Set $K:=C_0^{\beta_0}\subseteq C$.
By Corollary~\ref{lemma:ifinvariantsetemptythennointersection} and \eqref{defJt}, we have 
that  
\[
K=\Bigl\{ x\in K\mid \sum_{k=1}^\infty\Leb\bigl(B(x,r)\cap K(r_k)\bigr)=\infty 
\text{ for all } r>0\Bigr\}\ne\emptyset.
\]
Since $K\subseteq C$, it suffices to show that $\dimp K\ge d-s_2$. By a standard
argument, it is enough to prove that $\dimBu(V\cap K)\ge d-s_2$ for all open sets $V$ 
intersecting $K$. Indeed, since $\dimBu F=\dimBu\overline F$ for all bounded sets $F$, where
$\overline F$ is the closure of $F$, one may
use closed sets $F_i$ in \eqref{defdimMBu}. Then consider a covering of $K$ by closed sets 
 $\{F_i\}_{i=1}^\infty$, so that $K=\bigcup_{i=1}^\infty K\cap F_i$. Since $K$ 
is compact, it is of the second category. By the Baire's category theorem, there is an integer 
$i$ such that $K\cap F_i$ has nonempty interior relative to $K$, that is, there is an open set 
$V\subset X$ such that $\emptyset\ne V\cap K\subseteq K\cap F_i$.

Let $V\subseteq \R^d$ be an open set intersecting $K$. Then $V\cap K\supseteq B(x,r)\cap K$
for some $x\in K$ and $r>0$. Observe now that, for any $0<\delta<\frac r4$, 
\[
B(x,\tfrac r2)\cap K(\delta)\subseteq B(x,r-2\delta)\cap K(\delta)
\subseteq (B(x,r)\cap K)(\delta).
\]
Since $x\in K$, we have that 
\[
\sum_{k=1}^\infty \Leb\bigl(B(x,\tfrac r2)\cap K(r_k)\bigr)=\infty.
\]
Let $\varepsilon>0.$ Then  
\[
\Leb\bigl(B(x,\tfrac r2)\cap K(r_k)\bigr)\geq r_k^{s_2(1+\varepsilon)} 
\text{ for infinitely many } k\in \N,
\]
since otherwise, by \eqref{def-s2}, the above sum would be finite. Thus, 
\[
\Leb\bigl((V\cap K)(r_k)\bigr)\geq  \Leb\bigl((B(x,r)\cap K)(r_k)\bigr) 
\geq  r_k^{s_2(1+\varepsilon)}
\]
for infinitely many $k\in N$. Recall that the upper box counting dimension $\dimBu$ equals the 
upper Minkowski dimension $\dimMu$. Therefore,
\begin{align*}
\dimBu (V\cap K)&=\dimMu (V\cap K)=d-\liminf_{\delta \to 0}
      \frac{\log \Leb((V\cap K )(\delta))}{\log \delta}\\
    &\geq d-\liminf_{k\to \infty} \frac{\log \Leb((V\cap K )(r_k))}
      {\log r_k}\geq d-s_2(1+\varepsilon).
\end{align*}
Letting $\varepsilon\to 0$, we obtain $\dimBu (V\cap K)\geq d-s_2$, completing the 
proof of i).

\medskip

We   now prove iii),  which implies  ii) for compact sets. Suppose that
$C\cap E_{\underline{r}}(\omega)= \emptyset$ almost
surely. We will show that this implies the inequality $\dimMBl C\leq d-s_2$. 

By Lemma~\ref{lemma:Itildeinvariantsetfoundaftercountablymanysteps}, there exists a countable
ordinal $\beta_0$ such that $C_0^{\beta_0}=C_0^{\beta_0+1}$. By 
Proposition~\ref{lemma:conditionforpositivehittingprobability} applied to $C_0^{\beta_0}$, we
necessarily have that $C_0^{\beta_0}=\emptyset$. Thus, it suffices to show that 
\begin{equation}\label{equation:inductionclaimfordimMBl}
    \dimMBl (C\setminus C^{\lambda}_0)\leq d-s_2
\end{equation}
for each countable ordinal $\lambda$. To this end, we prove the following lemma.

\begin{lemma}
For every compact set $K\subseteq X$,  
\begin{equation}\label{equation:upperboundfordimMBl}
    \dimMBl (K\setminus \mathcal{J}_{\underline{r},0}^{{\Leb}}(K))\leq d-s_2.
\end{equation}
\end{lemma}

\begin{proof}[Proof of Lemma]
Since $\dimMBl$ is countably stable and
\begin{align*}
K\setminus \mathcal{J}_{\underline{r},0}^{{\Leb}}(K)&=\Bigl\{x\in K\mid \sum_{k=1}^\infty 
    \Leb\bigl(B(x,r)\cap K(r_k)\bigr)<\infty \text{ for some } r>0\Bigr\}\\
  &=\bigcup_{n=1}^\infty\Bigl\{x\in K\mid \sum_{k=1}^\infty 
    \Leb\bigl(B(x,\tfrac 1n)\cap K(r_k)\bigr)<\infty\Bigr\}\eqqcolon\bigcup_{n=1}^\infty W_n,
\end{align*}
it suffices to show that $\dimMBl W_n\leq d-s_2$ for every $n\in \N$. Fix $x\in W_n$ and 
$\varepsilon>0$. Note first that, for any $\delta>0$,
\[
\bigl(B(x,\tfrac 1{2n})\cap K\bigr)(\delta)\subseteq B\bigl(x,\tfrac 1{2n}+\delta\bigr)\cap K(\delta).
\]
Let $k_n\in\N$ be such that $r_k<\frac 1{2n}$ for all $k\ge k_n$. Then, for infinitely many 
$k\geq k_n$, we must have that 
\begin{equation}\label{equation:upperboundformeasuresofneighborhoods}
\Leb\bigl(\bigl(B(x,\tfrac 1{2n})\cap K\bigr)(r_k)\bigr) \leq\Leb\bigl(B(x,\tfrac 1{2n}+r_k)
\cap K(r_k)\bigr)< r_k^{s_2(1-\varepsilon)},
\end{equation}
since otherwise we would have that  
\begin{equation}
\label{sum-leb-infinite}
\sum_{k=1}^\infty \Leb\bigl(B(x,\tfrac 1n)\cap K(r_k)\bigr)\geq \sum_{k={k_n}}^\infty 
\Leb\bigl(B(x,\tfrac 1{2n}+r_k)\cap K(r_k)\bigr)=\infty,
\end{equation}
contradicting the fact that $x\in W_n$. Thus, 
\[
\frac{\log \Leb\bigl(\bigl(B(x,\tfrac 1{2n})\cap K\bigr)(r_k)\bigr)}{\log r_k}\geq 
s_2(1-\varepsilon)
\]
for infinitely many $k\in \N$. This yields
\begin{align}
\nonumber
\dimBl \bigl(B(x,\tfrac 1{2n})\cap K\bigr)& = \dimMl \bigl(B(x,\tfrac 1{2n})\cap 
  K\bigr)\\ \nonumber
&= d-\limsup_{\delta\to 0} \frac{\log \Leb\bigl((B(x,\tfrac 1{2n})
  \cap K )(\delta)\bigr)}{\log \delta}\\
& \leq d-\limsup_{k\to \infty} \frac{\log \Leb\bigl((B(x,\tfrac 1{2n})\cap K )
  (r_k)\bigr)}{\log r_k} \leq d-s_2(1-\varepsilon).
      \label{eq-maj-dim}
\end{align}
Since $\varepsilon>0$ was arbitrary, we obtain the inequality 
$\dimBl \bigl(B(x,\frac 1{2n})\cap K)\leq d-s_2$, and covering $W_n$ by finitely many balls
with centres in $W_n\subseteq K$ and radius $\frac 1{2n}$ yields 
\[
\dimMBl W_n \leq d-s_2.
\]
Thus \eqref{equation:upperboundfordimMBl} is established. \end{proof}

The general claim
\eqref{equation:inductionclaimfordimMBl} now follows by transfinite induction in the same way 
as \eqref{equation:inductionclaimforJt} was proved in the proof of
Proposition~\ref{lemma:necessaryoconditionfordimensionofintersection}.

\medskip

Item iv) can be proved in the same way as iii), but to get the packing dimension in the case 
$\underline r{(\alpha)}=(k^{-\alpha})_{k=1}^{\infty}$ instead of the modified 
lower box counting dimension in ii), the estimate
\eqref{equation:upperboundformeasuresofneighborhoods} must hold for all large $k\in\N$
instead of only infinitely many $k\in\N$.  In fact, if the opposite inequality is true for 
arbitrarily large $k\in\N$, one can show (as in the proof below starting from 
\eqref{equation:lowerbounforalphaobtainedfromalpha0} and concluding with
\eqref{sum-infinite})
that then the sum \eqref{sum-leb-infinite} will be infinite. This yields that we
can replace the lower box counting dimension by the upper one 
which, in turn, yields the upper bound for the modified upper box
counting dimension that coincides with the packing dimension.

\medskip

Finally,  we turn to the general case of analytic sets $C$. Claims ii) and iv) follow
from the fact that the Hausdorff and packing dimensions of analytic sets can be approximated
from inside by the ones of compact sets \cite{D1952,H1986,JP1995}. The claim i), in turn,
follows from Proposition~\ref{compacthitting}.
\end{proof}

\begin{remark}\label{newproof}
 We do not know whether part iii) of 
Theorem~\ref{theorem:hittingprobforregularmeasure} can be extended to analytic sets.
This  would follow if the modified lower box counting dimension could be approximated from inside
by compact sets, but we do not know if this is true. It is true for the modified
upper box counting dimension (since it equals the packing dimension) and it is not true for
the lower and upper box counting dimensions.   
\end{remark}

\section{Some remarks and perspectives}
\label{sec-conclusion}

\subsection {About the necessity of using ordinals}
\label{sec-ordinals}

The process described in 
Propositions~\ref{lemma:Itildeinvariantsetfoundaftercountablymanysteps-intro} and 
\ref{lemma:Itildeinvariantsetfoundaftercountablymanysteps} is
rather natural when studying hitting probabilities of random covering sets, and one might 
wonder whether a finite number of steps would be sufficient to reach the fixed point. In 
general, this is not true.

The following example (which we develop for simplicity in the case $t=0$) shows that the
invariant set in Proposition~\ref{lemma:Itildeinvariantsetfoundaftercountablymanysteps-intro} 
might not be found after iterating the map $\Jmurt$ over the natural numbers. The example is 
similar to examples of compact sets having Cantor--Bendixson rank larger than $\omega$, and 
can be adapted to obtain any countable  successor ordinal as $\beta_0$ in 
Proposition~\ref{lemma:Itildeinvariantsetfoundaftercountablymanysteps}. 
 
Let $\mc{L}_1$ denote the Lebesgue measure restricted to the unit interval $[0,1]$, and let 
$\underline{r}$ be a decreasing sequence of positive numbers such that 
\begin{equation}\label{equation:radiiforexample}
 \sum_{k=1}^\infty r_k<\infty \text{ and } \sum_{k=1}^\infty\sqrt{r_k}=\infty. 
\end{equation}
\begin{lemma} 
\label{lemma-nonfiniteordinal}
There exists a compact set $K\subseteq [0,1]$ such that 
\[
 \bigcap_{n=1}^\infty(\Jmurz)^n(K)=\{0\}.
\]
\end{lemma}

\begin{proof}
In order to define the desired set $K$, we first construct suitable compact sets 
$K_n\subseteq [0,1]$ for all $n\in \N$ by using induction. First, we let 
$K_1\coloneqq \{\frac{1}{n}\mid n\in \N \}\cup\{0\}.$ We now make some observations about 
$K_1$. If $0<\rho<\frac 14$, there is $k\in \N$ such that 
\[
\frac{1}{k(k+1)}= \frac{1}{k}-\frac{1}{k+1}<2\rho\leq \frac{1}{k-1}-\frac{1}{k}
 =\frac{1}{k(k-1)}.
\]
Then $k^{-1}\approx\sqrt{\rho}$ and $K_1(\rho)\supseteq [0,1/k[$, since 
$2\rho>\frac{1}{k}-\frac{1}{k+1}$. Thus, for any fixed $r>0$ and for every $k\in \N$ large
enough (depending on $r$), the intersection $B(0,r)\cap K_1(r_k)$ contains an interval with
length comparable to $\sqrt{r_k}$. By \eqref{equation:radiiforexample}, 
\begin{equation}\label{equation:0isinTildeJ1}
  \sum_{k=1}^\infty \mc{L}_1\bigl(B(0,r)\cap K_1(r_k)\bigr)=\infty 
\end{equation}
for every $r>0$. On the other hand, for any $x\in K_1\setminus \{0\}$ and for small enough 
$r>0$ (depending on $x$), we have that $B(x,r)\cap K_1\subseteq \{x\}$, hence, for $k$ large 
enough,
\[
B(x,r)\cap K_1(r_k)\subseteq B(x,r_k).
\]
By \eqref{equation:radiiforexample},
\begin{equation}\label{equation:not0isnotinTildeJ1}
  \sum_{k=1}^\infty \mc{L}_1\bigl(B(x,r)\cap K_1(r_k)\bigr)<\infty .
\end{equation}
Combining \eqref{equation:0isinTildeJ1} and \eqref{equation:not0isnotinTildeJ1}, we obtain 
that 
\begin{equation}\label{equation:tildeJ1is0}
  \Jmurz(K_1)=\{0\}.
\end{equation}
   
We now proceed with induction. Suppose that we have a compact set $K_n\subseteq [0,1]$ with 
the property that 
\[
(\Jmurz)^n(K_n)=\{0\}.
\]
We then define $K_{n+1}$ by setting
\[
K_{n+1}=\{0\}\cup \bigcup_{j=2}^{\infty}K_{j,n}, \ \ \mbox{ where }  \ \ 
K_{j,n}\coloneqq \frac{1}{j}+ \frac{1}{2j(j-1)} K_n .
\]
 Then, since 
\[
\dist(K_{j-1,n},K_{j,n})=\frac{1}{2}{\Bigl(\frac{1}{j-1}-\frac{1}{j}\Bigr)}  
\]
and $(\Jmurz)^{n}(K_{j,n})=\{\frac{1}{j}\}$ for every $j=2,3,\dots$, we have that 
\begin{equation}\label{intersectj}
(\Jmurz)^n(K_{n+1})\cap\Bigl[\frac 1j,\frac 1{j-1}\Bigr[=\Bigl\{\frac{1}{j}\Bigr\}
\end{equation}
for every $j=2,3,\dots$ Note also that the fact  
$(\Jmurz)^{n}(K_{j,n})=\{\frac{1}{j}\}\neq \emptyset$ implies that 
\[
\sum_{k=1}^\infty\mc{L}_1\bigl((\Jmurz)^{n-1}(K_{j,n})(r_k)\bigr)=\infty.
\]
For every $r>0$, we have $B(0,r)\supseteq K_{j,n}$ for some $j$, hence
\[
\sum_{k=1}^\infty \mc{L}_1\bigl(B(0,r)\cap ((\Jmurz)^{n-1}(K_{n+1}))(r_k)\bigr)=\infty.
\]
Thus, $0\in (\Jmurz)^n(K_{n+1})$. Combining this with equation \eqref{intersectj}, we 
conclude that $(\Jmurz)^n(K_{n+1})=K_1$. By \eqref{equation:tildeJ1is0}, 
\[
(\Jmurz)^{n+1}(K_{n+1})=\{0\}.
\]

By induction, for every $n=1,2,\dots$, there exists a compact set $K_n\subseteq [0,1]$ with 
the property that $(\Jmurz)^n(K_n)=\{0\}$. We can now define the compact set $K$ by setting 
\[
K\coloneqq \{0\}\cup \bigcup_{n=2}^{\infty}  \Bigl(\frac{1}{n}+\frac{1}{2n(n-1)}  K_n \Bigr).
\]
Then, for each $n=2,3,\dots$, 
\[
(\Jmurz)^n(K)\cap \Bigl[\frac 1n,\frac 1{n-1}\Bigr[=\Bigl\{\frac 1n\Bigr\}.
\]
This implies that $\{0\}{\subseteq}(\Jmurz)^n(K)$ for every $n$ and, by \eqref{equation:radiiforexample},
\[
(\Jmurz)^{n+1}(K)\cap \Bigl[\frac 1n,\frac 1{n-1}\Bigr[=\emptyset
\]
for every $n$. Thus, the set $K$ satisfies the claim of Lemma \eqref{lemma-nonfiniteordinal}.
\end{proof}
 
Now, it is immediate to check that the inequality in \eqref{equation:radiiforexample} implies 
that 
\[
\Jmurz\left(\{0\} \right)=\emptyset,
\]
since $\mc{L}_1(B(0,r))= r$ for every $0<r<1$. In particular, for the set $K$ built in 
Lemma \ref{lemma-nonfiniteordinal}, one has
\[
\Jmurz\left(\bigcap_{n=1}^\infty(\Jmurz)^n(K)\right)=\Jmurz\left(\{0\} \right)=\emptyset,
\]
which implies that $\beta_0=\omega+1$ for $K$.

\subsection{Final conclusions}

We finish the paper with some remarks and open questions:

$\bullet$ By modifying the weights used for the measures $\mu_k$ in  the 
construction of measures $\muOne$ and $\muTwo$ in
Sections~\ref{example:dimensionishausdorffspectrum} and 
\ref{example:dimensionispackingspectrum}, it is possible to build a measure $\mu$ such that   
any value between $\overline{F}_{\mu}(\frac 1\alpha)$ and 
$\min \{  \frac 1\alpha,\overline{H}_{\mu}(\frac 1\alpha)\}$ can be obtained for 
the value of $f_{\mu}(\alpha)$.

\medskip
$\bullet$ For the measure $\muOne$ in Section~\ref{example:dimensionishausdorffspectrum}, we 
had that $\dimlocu \muOne(x)=u$ for 
$x\in C$, whereas for $\muTwo$ in Section~\ref{example:dimensionispackingspectrum}, 
$\dimlocu \muTwo(x)=s$ for $x\in C$. One might wonder for a general measure $\mu$ if 
$f_{\mu}(\alpha)$ can be determined when knowing the
behaviours of both $\dimlocl \mu$ and $\dimlocu \mu$, but this turns out not to be the case.
By considering the measures 
\[
\mu:=\gamma_0\sum_{k=0}^{\infty}w_k\mu_k,
\]
where $\gamma_0$ is a normalising constant and
\[
w_k{:=}\begin{cases}
    2^{-\beta \frac{\log \ell_k}{\log a}}& \text{ if } \ell_k=b\ell_{k-1},\\
    2^{-\beta k}& \text{ if } \ell_k=a\ell_{k-1},
\end{cases}
\]
similar calculations as in Sections~\ref{example:dimensionishausdorffspectrum} and 
\ref{example:dimensionispackingspectrum} show that $\dimlocu \mu(x)= u$ for $x\in C$ so that 
the lower and upper local dimensions agree with the local dimensions of the measure 
$\muOne$ at every point, but the value of 
$f_{\mu}(\alpha)$ agrees with the value obtained for the measure $\muTwo$.

\medskip
$\bullet$ Given Theorem~\ref{theorem:dimhatleastconjecture} and 
\cite[Theorem 2.5]{JJMS2025}, it is expected that for a general sequence 
$\underline{r}$ and a general measure $\mu\in \mathcal{P}(\R^d)$, the value 
$f_\mu(\underline{r})$ can certainly be bounded in terms of $s_2(\underline{r})$, 
$\overline{F}_\mu$, $\overline{H}_\mu$, and the essential suprema of the local dimensions of 
$\mu$. But, as observed in \cite{JJMS2025}, these three parameters are not expected to be sufficient to determine the value of $\dim_H E_{\underline{r}}$.

\medskip
$\bullet$ The question of the nonemptiness of $E_{\alpha}(\omega)\cap D_\mu(t)$ when 
$\frac1\alpha =t-F_\mu(t)$ in Theorem~\ref{thm:dimhofintersectionwithsublevelset} certainly 
deserves to be studied.


\begin{thebibliography}{99}

\bibitem{BanicaVega} V. Banica and L. Vega, Riemann's non-differentiable function and the 
  binormal curvature flow, \textit{Arch. Ration. Mech. Anal.} \textbf{244} (2022), 501--540.
 
\bibitem{BarralFeng} J. Barral and D.-D. Feng, On multifractal formalism for self-similar 
  measures with overlaps, \textit{Math. Z.} \textbf{298} (1-2) (2021), 359--383. 
 
\bibitem{BS-Levy} J. Barral and S. Seuret, The singularity spectrum of L\'evy processes in 
  multifractal time, \textit{Adv. Math.} \textbf{214} (1) (2007), 437--468.
 
\bibitem{BS-Ubiq} J. Barral and S. Seuret, Heterogeneous ubiquitous systems in $\R^d$ and
  Hausdorff dimension, \textit{Bull. Braz. Math. Soc. (N.S.)} \textbf{38} (3) (2007),
  467--515. 

\bibitem{BV} V. Beresnevich and S. Velani, A mass transference principle and the 
  Duffin--Schaeffer conjecture for Hausdorff measures, \textit{Ann. of Math.} \textbf{164}
  (2006), 971--992.
 
\bibitem{Borel} E. Borel, Sur les s\'eries de Taylor, \textit{Acta Math.} \textbf{21} (1897),
  243--247.
 
\bibitem{Daviaud} E. Daviaud, A dimensional mass transference principle for Borel probability
  measures and applications, \textit{Adv. Math.} \textbf{474} (2025), Paper No. 110304, 47 pp. 

\bibitem{D1952} R. O. Davies, Subsets of finite measure in analytic sets,
  \textit{Indag. Math.} \textbf{14} (1952), 488--489.
  
\bibitem{Durand} A. Durand, On randomly placed arcs on the circle, 
  \textit{Recent developments in fractals and related fields}. Appl. Numer. Harmon. Anal.,   
  Birkhäuser Boston Inc., Boston, MA, 2010, 343--351.

\bibitem{Dvoretzky} A. Dvoretzky, On covering a circle by randomly placed arcs, 
  \textit{Proc. Nat. Acad. Sci. U.S.A} \textbf{42} (1956), 199--203.

\bibitem{EJJ2020} F. Ekstr\"om, E. J\"arvenp\"a\"a and M. J\"arvenp\"a\"a,
  Hausdorff dimension of limsup sets of rectangles in the Heisenberg group, 
  \textit{Math. Scand.} \textbf{126} (2020), 229--255.
  
\bibitem{EJJV2018}  F. Ekstr\"om, E. J\"arvenp\"a\"a, M. J\"arvenp\"a\"a and V. Suomala,   
  Hausdorff dimension of limsup sets of random rectangles in products of regular spaces, 
  \textit{Proc. Amer. Math. Soc.}  \textbf{146} (2018), 2509--2521.
  
\bibitem{EP2018} F. Ekstr\"om  and T. Persson, Hausdorff dimension of random
  limsup sets, \textit{J. London Math. Soc. (2)} \textbf{98} (2018), 661--686.
 
\bibitem{E-B2024} S. Eriksson-Bique, A new Hausdorff content bound for limsup
  sets, \textit{Adv. Math.} \textbf{445} (2024), Paper No. 109638, 52 pp. 

\bibitem{Fal2003} K. Falconer, Fractal Geometry, John Wiley \& Sons, Inc., Hoboken, NJ, 2003.  

\bibitem{Fan-Kahane} A.-H. Fan and J.-P. Kahane, Raret\'e des intervalles recouvrant un point
  dans un recouvrement al\'eatoire, \textit{Ann. Inst. Henri Poincar\'e Prob. Stat.} 
  \textbf{29} (3) (1993), 453--466.

\bibitem{Fan2} A.-H. Fan and D. Karagulyan,  On $\mu$-Dvoretzky random covering of the  
  circle, \textit{Bernoulli} \textbf{27} (2) (2021), 1270--1290. 
  
\bibitem{FST}  A.-H. Fan, J. Schmeling and  S. Troubetzkoy, A multifractal mass transference 
  principle for Gibbs measures with applications to dynamical Diophantine approximation, 
  \textit{Proc. Lond. Math. Soc.} \textbf{107} (2013), 1173--1219.

\bibitem{Fan-Wu} A.-H. Fan  and J. Wu, On the covering by small random intervals, 
  \textit{Ann. Inst. Henri Poincar\'e Probab. Stat.} \textbf{40} (2004), 125--131.

\bibitem{FJJV} D.-J. Feng, E. J\"arvenp\" a\" a, M. J\"arvenp\"a\"a and V. Suomala, 
  Dimensions of random covering sets in Riemann manifolds, \textit{Ann. Prob.} \textbf{46}
  (2018), 1542--1596.

\bibitem{H1986} H. Haase, Non-$\sigma$-finite sets for packing measure,  
  \textit{Mathematika} \textbf{33} (1986), 129--136.
   
\bibitem{Hill-Velani} R. Hill and S. Velani, The ergodic theory of shrinking targets, 
  \textit{Invent. Math.} \textbf{119} (1995), 175--198.

\bibitem{H1995} J. D. Howroyd, On dimension and on the existence of sets of finite positive
  Hausdorff measure, \textit{Proc. London Math. Soc.} \textbf{70} (1995), no. 3, 581--604.
 
\bibitem{HLX} Z.-N. Hu, B. Li and Y. Xiao, On the intersection of dynamical covering sets with
  fractals, \textit{Math. Z.} \textbf{301} (2022), no. 1, 485--513.
  
\bibitem{Jaffard-levy} S. Jaffard, The multifractal nature of L\'evy processes,
  \textit{Probab. Theory Relat. Fields} \textbf{114} (1999), 207--227. 

\bibitem{Jaffard-lac} S. Jaffard, On lacunary wavelet series, 
  \textit{Ann. Appl. Probab.} \textbf{10} (1) (2000), 313--329. 

\bibitem{JaffardMartin} S. Jaffard and B. Martin, Multifractal analysis of the Brjuno 
  function, \textit{Invent. Math.} \textbf{212} (2018), 109--132. 

\bibitem{JJKLS} E. J\"arvenp\"a\"a, M. J\"arvenp\"a\"a, H. Koivusalo,
  B. Li  and V. Suomala,  Hausdorff dimension of affine random covering sets in torus,
  \textit{Ann. Inst. Henri Poincar\'e Probab. Stat.} \textbf{50} (2014), 1371--1384.

\bibitem{JJKLSX2017} E. J\"arvenp\"a\"a, M. J\"arvenp\"a\"a, H. Koivusalo,
  B. Li, V. Suomala and Y. Xiao, Hitting probabilities of random covering sets
  in tori and metric spaces, \textit{Electron. J. Probab.} \textbf{22} (2017),
  paper no. 1, 18 pp.
  
\bibitem{JJMS2025} E. J\"arvenp\"a\"a, M. J\"arvenp\"a\"a, M. Myllyoja and \"O. Stenflo,
  The Ekstr\"om--Persson conjecture regarding random covering sets, 
  \textit{J. London Math. Soc.} \textbf{111} (2025), paper no. e70058, 27pp.

\bibitem{JP1995} H. Joyce and D. Preiss, On the existence of subsets of finite 
  positive packing measure, \textit{Mathematika} \textbf{42} (1995), 15--24.

\bibitem{K1995} A. Kechris, Classical Descriptive Set Theory, Springer-Verlag, New York, 1995.

\bibitem{KPX2000} D. Khoshnevisan, Y. Peres and Y. Xiao, Limsup random fractals,
  \textit{Electron. J. Probab.} \textbf{5} (2000), paper no. 5, 24 pp.

\bibitem{KR} H. Koivusalo and M. Rams, Mass transference principle: from balls to arbitrary 
  shapes, \textit{Int. Math. Res. Not. IMRN} \textbf{8} (2021), 6315--6330.

\bibitem{ViklundLawler} G. F. Lawler and F. J. Viklund, Almost sure multifractal spectrum for
  the tip of an SLE curve, \textit{Acta Math.} \textbf{209} (2) (2012), 265--322.

\bibitem{Liao-Velani-Zorin} B. Li, L. Liao, S. Velani and E. Zorin, The shrinking target 
  problem for matrix transformations of tori: revisiting the standard problem, 
  \textit{Adv. Math.} \textbf{421} (2023), Paper No. 108994.

\bibitem{LS2014} B. Li and V. Suomala, A note on the hitting probabilities of random covering
  sets, \textit{Ann. Acad. Sci. Fenn. Math.} \textbf{39} (2014), no. 2, 625--633.
  
\bibitem{Liao-Seuret} L. Liao and S. Seuret, Diophantine approximation by orbits of expanding 
  Markov maps, \textit{Ergodic Theory Dynam. Systems} \textbf{33} (2013), 585--608.

\bibitem{Mat1995} P. Mattila, Geometry of Sets and Measures in Euclidean Spaces.
  Fractals and rectifiability, Cambridge University Press, Cambridge, 1995.  

\bibitem{Myllyoja1} M. Myllyoja, Hausdorff dimension of limsup sets of isotropic rectangles 
  in Heisenberg groups, \textit{J. Fractal Geom.} \textbf{11} (2024), 219--246.

\bibitem{Persson} T. Persson, A note on random coverings of tori, 
  \textit{Bull. Lond. Math. Soc.} \textbf{47} (2015), 7--12.

\bibitem{Seuret} S. Seuret, Inhomogeneous random coverings of topological Markov shifts,
  \textit{Math. Proc. Cam\-bridge Philos. Soc.} \textbf{165} (2018), 341--357.

\bibitem{Shepp} L. Shepp, Covering the circle with random arcs, \textit{Israel J. Math.}
  \textbf{11} (1972), 328--345.

\bibitem{Shmerkin} P. Shmerkin, On Furstenberg's intersection conjecture, self-similar
  measures, and the $l^q$ norms of convolutions, \textit{Ann. of Math.} \textbf{189} (2) 
  (2019), 319--391.

\end{thebibliography}
\end{document}